\pdfoutput=1
\documentclass{amsart}
\usepackage{amsmath}
\usepackage{amsthm}
\usepackage{amsfonts}
\usepackage{mathtools}
\usepackage{fullpage}
\usepackage{amssymb}

\theoremstyle{theorem}
\newtheorem{theorem}{Theorem}
\newtheorem{lemma}{Lemma}
\newtheorem{proposition}{Proposition}

\newtheorem{conjecture}{Conjecture}
\newtheorem{corollary}{Corollary}
\newtheorem{question}{Question}
\theoremstyle{definition}
\newtheorem{remark}{Remark}
\newtheorem*{notation}{Notation}
\theoremstyle{definition}
\newtheorem*{definition}{Definition}

\title[three]{
	Coefficients of the monodromy matrices\\
	of one-parameter families of double octic Calabi-Yau threefolds\\
	at a half-conifold point}
\author{Tymoteusz Chmiel}
\address{Jagiellonian University,
ul. {\L}ojasiewicza 6,
30-348 Krak\'ow,
Poland}
\begin{document}
\maketitle
\vspace{-8mm}
\begin{abstract}
Doran and Morgan introduced in \cite{Doran-Morgan} a rational basis for the monodromy group of the Picard-Fuchs operator of a hypergeometric family of Calabi-Yau threefolds. In this paper we compute numerically the transition matrix
between a generalization of the Doran-Morgan basis and the Frobenius basis at
a half-conifold point of a one-parameter family of double octic
Calabi-Yau threefolds. We identify the entries of this matrix as
rational functions in the special values $L(f,1)$ and $L(f,2)$ of the corresponding modular form $f$ and one constant. We also present related results concerning the rank of the group of period integrals generated by the action of the monodromy group on the conifold period.
\end{abstract}

\section*{Introduction}

The Picard-Fuchs equation is a differential equation satisfied by periods of a family of algebraic manifolds. Since the fundamental paper \cite{Candelas-de la Ossa-Green-Prakes} Picard-Fuchs operators associated with one-dimensional families of Calabi-Yau threefolds have been an object of intensive research, partially due to their connections with mirror symmetry. One of the most appealing aspects of this connection is the observation that in many cases numerical invariants of a Calabi-Yau threefold can be identified using the Picard-Fuchs operator of a mirror family. It is conjectured that they appear as coefficients in the transition matrix $T^\mathcal{B}_{F_m}$ between certain rational basis of solutions $\mathcal{B}$ and the Frobenius basis $F_m$ at a point of maximal unipotent monodromy. Furthermore, these conjectures imply that with respect to the basis $F_m$ the monodromy group is a subgroup of $GL(4,\mathbb{Q}(\tfrac{\zeta(3)}{2\pi i})]$.

We present results of similar type concerning the coefficients of the transition matrix $T^\mathcal{B}_{F_c}$ between the Doran-Morgan basis $\mathcal{B}$ and the Frobenius basis $F_c$ at a conifold singularity $s$. The former is a global, rational basis, while the latter is defined locally by the Frobenius method. We were able to recognize coefficients of the transition matrix $T^\mathcal{B}_{F_c}$, up to one unidentified constant, in the case of half-conifold singularities appearing in families of double octic Calabi-Yau threefolds. We also discuss an example of a singularity of type $\tfrac{1}{4}$C; in this case we were able to identify the transition matrix completely. As such, our results propose a potential candidate for a field $K$ such that the monodromy matrices with respect to $F_c$ lie in $GL(4,K)$.

Both a motivation and a way of approaching this problem is provided by our results presented in \cite{Chmiel}. There we use Picard-Fuchs operator $\mathcal{P}$ for numerical identification of period integrals of certain rigid Calabi-Yau threefolds. The method requires the existence of a conifold point $s$ at which the family degenerates to a singular model of a rigid Calabi-Yau manifold $X$. In Section \ref{s:values}, using analytic continuation of the conifold period around a MUM singularity, we introduce a group $\mathcal{L}^0(\mathcal{P},s)$. Up to commensurability, the lattice of period integrals of $X$ is a subgroup of $\mathcal{L}^0(\mathcal{P},s)$ and if $\operatorname{rank}(\mathcal{L}^0(\mathcal{P},s))=2$, they are equal. In several cases $\operatorname{rank}(\mathcal{L}^0(\mathcal{P},s))=3$ and period integrals of the singular fiber can be identified as well.

In this paper we shall study two natural questions arising in this situation. Numerical data suggests that if $\mathcal{L}^0(\mathcal{P},t_0)$ is of rank $2$, the special values $L(X,1)$ and $\tfrac{L(X,2)}{2\pi i}$ of the $L$-function of $X$ generate $\mathcal{L}^0(\mathcal{P},s)\otimes\mathbb{Q}$. Therefore this case is of special interest: can we find some conditions on the singularity $s$ which guarantee that $\operatorname{rank}(\mathcal{L}^0(\mathcal{P},s))=2$? Furthermore, assume that instead of considering monodromies around a MUM point, we consider the action of the entire fundamental group on the conifold period. This leads to a definition of $\mathcal{L}(\mathcal{P},s)$. Is the inclusion $\mathcal{L}^0(\mathcal{P},s)\subset\mathcal{L}(\mathcal{P},s)$ strict? In particular, if $\operatorname{rank}(\mathcal{L}^0(\mathcal{P},t_0))=2$ and it is generated by periods of $X$, can the period integrals of the singular fiber still be identified by considering $\mathcal{L}(\mathcal{P},s)$?

It turns out that by identifying the form of the transition matrix $T^\mathcal{B}_{F_c}$ both of these questions can be answered, albeit not in full generality. Conifold singularities are characterised by the local exponents $(0,1,1,2)$. In this paper we consider more special case of \textit{half-conifold points} with local exponents $(0,\tfrac{1}{2},\tfrac{1}{2},1)$. Via a quadratic pull-back, half-conifold singularities can be considered more symmetric counterparts of the conifold points. Local solutions near half-conifold points behave more controllably with respect to the action of the monodromy group and this allows us to answer the considered questions. In particular, we show that for half-conifold singularities the rank of the group $\mathcal{L}(\mathcal{P},s)$ is (at most) 2 and deduce from that the $\mathbb{Q}$-vector spaces generated by elements of $\mathcal{L}(\mathcal{P},s)$ and $\mathcal{L}^0(\mathcal{P},s)$ are equal.

The struture of this paper is as follows. In the first section we introduce Picard-Fuchs operators and their monodromy representations. Section \ref{s:DM} presents a construction of the generalized Doran-Morgan basis, which is a rational basis for the action of the monodromy group. Sections \ref{s:trans} and \ref{s:frob} are concerned with local bases of solutions: the former contains definitions of singularities we consider, while the latter focuses on the form of the Frobenius bases at these singular points. In section \ref{s:mod} we present basic facts concerning modular forms and the modularity theorem for Calabi-Yau threefolds. In section \ref{s:values} we recall results form \cite{Chmiel} and formulate the problems we consider. The last three sections contain main results of the paper. Section \ref{s:theorem} contains a proof of the theorem concerning the dimension of the $\mathbb{Q}$-vector space spanned by the values of the conifold period after the monodromy action. Section \ref{s:coeff} presents results on the exact values of the coefficients of the transition matrix $T^\mathcal{B}_{F_c}$ between the Doran-Morgan basis and the Frobenius basis at a half-conifold point, while the section \ref{s:14} deals with the special case of a $\tfrac{1}{4}$C singularity. The Appendix contains precise form of the results presented in section \ref{s:coeff}.

\section{Monodromy representation of a Picard-Fuchs operator}\label{s:PF_and_mon}

A \textit{Calabi-Yau manifold} is a smooth complex projective variety of dimension $d$ such that $\omega_X\simeq\mathcal{O}_X$ and $H^i(X,\mathcal{O}_X)=0$ for $i=1,\dots,d-1$. Basic invariants of any complex projective manifold are its \textit{Hodge numbers} $h^{i,j}(X):=\text{dim}(H^j(X,\Omega^i_X))$. They are often presented in the form of a diagram called the \textit{Hodge diamond} of $X$. In this paper we will focus on the case of Calabi-Yau threefolds. For such a manifold the Hodge diamond has the following form:
$$
\begin{matrix}
1\\
0\quad 0\\
0\quad h^{1,1}\quad 0\\
1\quad h^{1,2}\quad h^{1,2}\quad 1\\
0\quad h^{1,1}\quad 0\\
0\quad 0\\
1\\
\end{matrix}
$$
The Hodge number $h^{1,1}(X)$ is equal to the Picard number $\rho(X)$, while by the Bogomolov-Tian-Todorov unobstructedness theorem $h^{1,2}(X)$ is the dimension of a smooth deformation space (the \textit{Kuranishi space}) of the manifold $X$.

When $h^{1,2}(X)=0$, the Calabi-Yau threefold $X$ has no deformations of the complex structure and thus is called \textit{rigid}. When $h^{1,2}(X)=1$, there is a flat proper morphism $\mathfrak{X}\rightarrow S$ onto a smooth curve $S$ with a distinguished point $t_0\in S$ and a fixed isomorphism $X_{t_0}\simeq X$. In this situation we say that $X$ belongs to a \textit{one-dimensional family} $\mathfrak{X}=(X_t)_{t\in S}$.

Let $X\simeq X_{t_0}$ be a Calabi-Yau threefold deforming in a one-dimensional family. Take a family of holomorphic $3$-forms $\omega_t\in H^{3,0}(X_t)$, $t\in S$, depending holomorphically on $t$, and a smooth family of $3$-cycles $\gamma_t\in H_3(X_t,\mathbb{Z})$. Given such data we can define so called \textit{period functions}:
$$y(t):=\displaystyle\int_{\gamma_t}\omega_t,$$
which play crucial role in studying the family $\mathfrak{X}$. Since the space $H^3(X_{t_0},\mathbb{C})$ has dimension $4$, the elements
$$\left(\nabla_{\frac{\partial}{\partial t}}^4\omega_{t}\right)\Big\rvert_{t_0},\ 
\left(\nabla_{\frac{\partial}{\partial t}}^3\omega_{t}\right)\Big\rvert_{t_0},\ 
\left(\nabla_{\frac{\partial}{\partial t}}^2\omega_{t}\right)\Big\rvert_{t_0},\ 
\left(\nabla_{\frac{\partial}{\partial t}}\omega_t,\right)\Big\rvert_{t_0},\ 
\omega_{t_0}$$
are linearly dependent. As a consequence, the period function $y$ satisfies a differential equation $\mathcal{P}y=0$ for some fourth order differential operator $\mathcal{P}$, called the \textit{Picard-Fuchs operator} of the family $\mathfrak{X}$.

The Picard-Fuchs operator has only regular singularities, i.e. it is a Fuchsian operator. Crucial role in the study of Fuchsian operators is played by the \textit{monodromy representation}, which we now briefly recall. Since in this paper we  consider only the case on families over
Zariski-open sets in $\mathbb{P}^1$, for simplicity we shall restrict our
considerations to this situation.

Let $\mathcal{P}$ be a Fuchsian differential operator of order $4$ and let $\mathcal{S}\subset\mathbb{P}^1$ denote the set of its singular points. Let $V=\mathcal{S}ol(t_0)$ be the space of solutions of $\mathcal{P}=0$ in a small neighbourhood of some point $t_0\not\in\mathcal{S}$. Any solution of $\mathcal{P}=0$ can be continued analytically along any loop $\gamma\in\pi_1(\mathbb{P}^1\setminus\mathcal{S},t_0)$ and the result is still a solution. Since continuations along homotopic paths are equal, we obtain a representation $M:\pi_1(\mathbb{P}^1\setminus\mathcal{S},t_0)\rightarrow GL(V)$, called \textit{the monodromy representation}, whose image is the monodromy group $Mon(\mathcal{P})$. After choosing a basis $\mathcal{B}$ of $V$ and thus fixing an isomorphism $V\simeq\mathbb{C}^4$, we obtain an associated representation $M^{\mathcal{B}}:\pi_1(\mathbb{P}^1\setminus\mathcal{S},t_0)\rightarrow GL(4,\mathbb{C})$ which identifies $Mon(\mathcal{P})$ with a subgroup of $GL(4,\mathbb{C})$.

To compute the monodromy group exactly seems to be very difficult problem and has only been done in a few special cases (\cite{Doran-Morgan}, \cite{Chen-Yang-Yui}). Thus one usually relies on numerical approximations. In such situation it is very convenient to have a basis $\mathcal{B}$ of $\mathcal{S}ol(t_0)$ such that $Mon^\mathcal{B}(\mathcal{P})\subset GL(4,\mathbb{Q})$, which makes numerical identification of the entries of the monodromy matrices much easier. In the next section we shall describe an example of such a rational basis, the \textit{(generalized) Doran-Morgan basis}.

\begin{notation}
$Mon(\mathcal{P})$ denotes the monodromy group of $\mathcal{P}$ with respect to a fixed base point $t_0$, which we suppress from the notation. For any $\gamma\in\pi_1(\mathbb{P}^1\setminus\mathcal{S},t_0)$ we denote by $M_\gamma$ the action of monodromy along the loop $\gamma$. If we choose a basis $\mathcal{B}$ of $\mathcal{S}ol(t_0)$, we denote the corresponding monodromy group, which is now a subgroup of $GL(4,\mathbb{C})$, by $Mon^\mathcal{B}(\mathcal{P}):=\operatorname{im}(M^\mathcal{B})$. Similarly if $A\in Mon(\mathcal{P})$, we denote by $A^\mathcal{B}\in GL(4,\mathbb{C})$ the matrix of $A$ with respect to $\mathcal{B}$. By $T^{\mathcal{A}}_{\mathcal{B}}$ we denote the transition matrix between two bases $\mathcal{A}$ and $\mathcal{B}$. Finally, by $e_k$, $k=1,\cdots,4$, we denote the $k$th standard basis vector of $\mathbb{C}^4$.
\end{notation}

\section{Generalized Doran-Morgan basis}\label{s:DM}

Let $\mathcal{P}$ be a differential operator of order 4 and assume that there exists an element $M\in Mon(\mathcal{P})$ such that $(M-\operatorname{Id})^4=0$ but $(M-\operatorname{Id})^3\neq0$. Such an element will be called \textit{maximally unipotent}. In particular, if we put $N:=M-\operatorname{Id}$, then $\text{dim}(\text{ker}(N))=1$. Pick a solution $v\in\mathcal{S}ol(t_0)$ which is a generator of $\text{ker}(N)$. For any $w\in\mathcal{S}ol(t_0)$ we have $N^4(w)=0$, so that $N^3(w)\in\text{ker}(N)=\text{span}(v)$. Denote by $d(w)$ the uniquely determined complex number such that $N^3(w)=d(w)\cdot v$. For $A\in Mon(\mathcal{P})$ put $B:=A-\operatorname{Id}$ and define $d_A:=d(B(-v))$.

\begin{definition}
With the notation as above assume that $d_A\neq 0$. The \textit{generalized Doran-Morgan basis for $M$ and $A$} is the following set of solutions defined in a neighbourhood of a fixed base point $t_0$:
\begin{equation*}
\mathcal{B}_{M,A}:=\left\{B(-v),N(B(-v)),\frac{N^2(B(-v))}{d_A},v\right\}.
\end{equation*}
\end{definition}

Note that if $d_A=0$ for all $A\in Mon(\mathcal{P})$, then $N^3(A(v))=0$ for all $A\in Mon(\mathcal{P})$. Let $W:=\text{span}\big(\{A(v):A\in Mon(\mathcal{P})\}\big)$. It is easily seen that $W$ is a non-zero monodromy-invariant subspace of $\mathcal{S}ol(t_0)$. Since $W\subset\text{ker}(N^3)$, it is a proper subspace as well. Thus if the considered local system is irreducible, we can always find $A\in Mon(\mathcal{P})$ such that $d_A\neq0$ and consequently construct a generalized Doran-Morgan basis.

It is straightforward to verify the following:
\begin{proposition}
	$\mathcal{B}_{M,A}$ is a basis of solutions of $\mathcal{P}$ near $t_0$.
\end{proposition}

\begin{remark}\label{remark:MUM}
	In the generalized Doran-Morgan basis $\mathcal{B}_{M,A}$ action of the monodromy operator $M$ has matrix of the following form:
	\begin{equation*}
	M^{\mathcal{B}_{M,A}}=\begin{pmatrix}
	1&0&0&0\\1&1&0&0\\0&d_A&1&0\\0&0&1&1\\
	\end{pmatrix}
	\end{equation*}
\end{remark}

After this general construction, let us now return to the situation when $\mathcal{P}$ is a Picard-Fuchs operator of a family of Calabi-Yau threefolds given by a morphism $\varphi:\mathcal{X}\rightarrow\mathbb{P}^1\setminus\mathcal{S}$ with fibers $X_t$. The local system associated to the differential equation $\mathcal{P}=0$, with fiber $\mathcal{S}ol(t)$ over $t\in\mathbb{P}^1\setminus\mathcal{S}$, is isomorphic to the local system $R^3\varphi_*\mathbb{C}$. The latter denotes the direct image of a constant sheaf $\mathbb{C}$ on the total space $\mathcal{X}$ and its fiber over $t\in\mathbb{P}^1\setminus\mathcal{S}$ is the cohomology group $H^3(X_t,\mathbb{C})$.

The local system $R^3\varphi_*\mathbb{C}$ contains a local system of rational vector spaces $R^3\varphi_*\mathbb{Q}$ with fibers $H^3(X_t,\mathbb{Q})$ such that $H^3(X_t,\mathbb{Q})\otimes\mathbb{C}=H^3(X_t,\mathbb{C})$. Fix a base point $t_0\not\in\mathcal{S}$. By the Ehresmann lemma the family $X_t$ is smoothly locally trivial and thus the subspace $H^3(X_{t_0},\mathbb{Q})$ is monodromy-invariant. We conclude that there exists a rational subspace $\mathcal{S}ol_{\mathbb{Q}}(t_0)\subset\mathcal{S}ol(t_0)$, which is preserved under the action of the monodromy group, such that $\mathcal{S}ol_{\mathbb{Q}}(t_0)\otimes\mathbb{C}=\mathcal{S}ol(t_0)$. Clearly with respect to any basis $\mathcal{B}$ of $\mathcal{S}ol_{\mathbb{Q}}(t_0)$ the monodromy group $Mon^{\mathcal{B}}(\mathcal{P})$ is a subgroup of $GL(4,\mathbb{Q})$.

Consider the generalized Doran-Morgan basis $\mathcal{B}_{M,A}$ for appropriate $M,A\in Mon(\mathcal{P})$. Its construction depends on the choice of the initial solution $v$. However, any two such choices differ by a multiplicative constant. Since other elements of the basis $\mathcal{B}_{M,A}$ are obtained from $v$ via the action of certain linear operators, it follows that the group $Mon^{\mathcal{B}_{M,A}}(\mathcal{P})$ is independent of the choice of $v$. Pick $0\neq v\in\mathcal{S}ol_\mathbb{Q}(t_0)\cap\operatorname{ker}(N)$. Since $\mathcal{S}ol_\mathbb{Q}(t_0)$ is monodromy-invariant, all other elements of $B_{M,A}$ will be elements of $\mathcal{S}ol_\mathbb{Q}(t_0)$ as well. Since as a basis of $\mathcal{S}ol(t_0)$ they are linearly independent, we see that for this choice of $v$ the generalized Doran-Morgan basis $B_{M,A}$ forms a basis of $\mathcal{S}ol_\mathbb{Q}(t_0)$. This way we obtain:

\begin{proposition}
The generalized Doran-Morgan basis is rational, i.e. $Mon^{\mathcal{B}_{M,A}}(\mathcal{P})\subset GL(4,\mathbb{Q})$.
\end{proposition}

\noindent In particular we conclude that $d_A\in\mathbb{Q}$; in fact, using the existence of monodromy-invariant fiberwise sublattice $R^3\varphi_*\mathbb{Z}$, one can show that $d_A\in\mathbb{Z}$.

\section{Local monodromy operators}\label{s:trans}

The construction of the generalized Doran-Morgan basis can be described purely in terms of linear algebra and is rather straightforward. The only problem is finding a maximally unipotent element $M$ in the monodromy group $Mon(\mathcal{P})$. An element is maximally unipotent if and only if it has one Jordan block with eigenvalue $1$. However, given a differential equation and a loop omitting its singular point, it is not obvious how to determine the Jordan form of the associated monodromy transformation. Nevertheless, there are elements of the monodromy group $Mon(\mathcal{P})$ for which there exists an easy algorithm which determines their Jordan form. These are the elements defining so called \textit{local monodromies}; they are the images of the standard generators of the fundamental group of $\mathbb{P}^1\setminus{S}$, i.e. the loops encircling the singular points.

Let $s\in\mathcal{S}$ be a singular point of a differential operator $\mathcal{P}$ and let $\gamma\in\pi_1(\mathbb{P}^1\setminus\mathcal{S},t_0)$ be any loop such that the winding number of $\gamma$ around $s$ is 1 and the winding number of $\gamma$ around any other element of $\mathcal{S}$ is 0. In this situation we call $M_\gamma$ a \textit{local monodromy operator} of $s$. Of course $M_\gamma$ depends not only on $s$ but also on the choice of $\gamma$. However, any two such choices are conjugate in $\pi_1(\mathbb{P}^1\setminus\mathcal{S},t_0)$ and hence the Jordan form of a local monodromy is a well-defined invariant of the singularity $s$. When it does not lead to a confusion, we denote by $M_s$ a local monodromy around $s$.

The eigenvalues of a local monodromy around any singularity $s$ can be easily read off the differential equation itself. To do this use an automorphism of $\mathbb{P}^1$ to move the singularity $s$ to $0$ and then write the operator $\mathcal{P}$ using the logarithmic derivative $\Theta:=t\tfrac{d}{dt}$:
$$
\mathcal{P}=\Theta^4+q_1(t)\Theta^3+q_2(t)\Theta^2+q_3(t)\Theta+q_4(t),\quad q_i\in\mathbb{C}(t)
$$
For a Fuchsian operator the coefficients $q_i$ are holomorphic at $0$ and we can read the Jordan form of the local monodromy operator $M_0$ from the \textit{indicial equation}:
$$
X^4+q_1(0)X^3+q_2(0)X^2+q_3(0)X+q_4(0)=0
$$

More precisely, let $\alpha_1,\cdots,\alpha_4$ be the roots of the indicial equation. Then the eigenvalues of the local monodromy at $s$ are $\exp(2\pi i\alpha_1),\cdots,\exp(2\pi i\alpha_4)$. For Picard-Fuchs operators, if $\alpha$ is a root of the indicial equation of multiplicity $m$, the Jordan block associated to $\exp(2\pi i\alpha)$ has size $m$ and distinct roots correspond to different Jordan blocks (cf. \cite{Hofmann}, Theorem 4.3.6). Thus we see that identifying the Jordan form of a local monodromy operator is a direct computation.

\begin{definition}
	A singularity $s\in\mathcal{S}$ of a Fuchsian differential operator of order 4 is called a \textit{point of maximal unipotent monodromy} or a \textit{MUM point} if the local monodromy at $s$ has Jordan form:
	$$\begin{pmatrix}
	1&1&0&0\\0&1&1&0\\0&0&1&1\\0&0&0&1\\
	\end{pmatrix}$$
\end{definition}

Assume that $\mathcal{P}$ has a MUM point at 0. By the very definition we may take any local monodromy $M_0:=M_\gamma$ around $0$ as $M$ in the construction of the generalized Doran-Morgan basis $\mathcal{B}_{M,A}$. Usually the most natural choice is to pick the base point $t_0$ such that $|t_0|\ll 1$ and a loop $\gamma$ such that $\operatorname{im}(\gamma)$ is contained in a small disc centered of 0 which contains no other singularities. This procedure allows us to construct generalized Doran-Morgan basis in a consistent way for a wide range of operators.

In the rest of this paper we shall consider only operators which have a point of maximal unipotent monodromy. Apart from providing natural rational basis of solutions, the existence of a MUM point is also crucial from the point of view of the mirror symmetry and recovering arithmetical data. For example in \cite{van Straten} the so-called \textit{Calabi-Yau operators}, which provide an abstract version of Picard-Fuchs operators, are required to have a MUM point. Nevertheless, there are numerous examples of Picard-Fuchs operators which do not have a MUM point (see \cite{Cynk-van Straten_2}). In general, it is not clear how to approach the construction of rational basis in those cases.

Another class of singularities which is crucial for our further investigations are singularities of type $\tfrac{1}{n}C$. As a local monodromy around a MUM point provides a natural candidate for the operator $M$ in the construction of $\mathcal{B}_{M,A}$, local monodromy around $\tfrac{1}{n}C$ will serve as $A$. This kind of singularities is particularly interesting due to their conjectural relations with mirror symmetry and modularity. The reason for our focus on them in this paper is the former.

\begin{definition}
	A singularity $s\in\mathcal{S}$ of a Fuchsian differential operator of order 4 is called a \textit{singularity of type} $\tfrac{1}{n}C$ if the local monodromy at $s$ has Jordan form:
	\begin{equation*}
	\begin{pmatrix}
	1&0&0&0\\0&\zeta_n&1&0\\0&0&\zeta_n&0\\0&0&0&\zeta_{n}^2\\
	\end{pmatrix}
	\end{equation*}
	where $\zeta_n$ is some primitive  $n$th root of unity. When $n=1$, resp. $n=2$, we use the name \textit{conifold}, resp. \textit{half-conifold, singularities}.
\end{definition}

Doran-Morgan bases were introduced in \cite{Doran-Morgan}, in the case of \textit{hypergeometric operators}. These operators have three singular points at $0$, $1$ and $\infty$; furthermore, $0$ is a MUM point and $1$ is a conifold point. In this situation local monodromies $M_0$ and $M_1$ generate the entire monodromy group. If we consider the Doran-Morgan basis $\mathcal{B}:=\mathcal{B}_{M_0,M_1}$, one easily checks that
\begin{equation*}
M_1^{\mathcal{B}}=
\begin{pmatrix}
1&-K&-1&-1\\0&1&0&0\\0&0&1&0\\0&0&0&1\\
\end{pmatrix}
\end{equation*}

For general $\tfrac{1}{n}C$ singularities we consider a variant of this construction:
\begin{definition}
Let $\mathcal{P}$ be a differential operator of order 4 with a MUM point at 0 and a $\tfrac{1}{n}C$ singularity at $s$. Let $\gamma$, resp. $\delta\in\pi_1(\mathbb{P}^1\setminus\mathcal{S},b)$, be a loop giving local monodromy around 0, resp. $s$. Put $M:=M_\gamma$, $A:=M_\delta^n$ and assume that $d_A\neq 0$ with respect to $M$. The \textit{Doran-Morgan basis} associated to a couple $(\mathcal{P},s)$ is $\mathcal{B}_{M,A}$.
\end{definition}

Let $s$ be $\tfrac{1}{n}C$ singularity of an operator $\mathcal{P}$. By inspecting the Jordan form of $M_{\delta^n}=M_\delta^n$, we see that it is the same as that of a local monodromy around an ordinary conifold point. 
Thus for the Doran-Morgan basis $\mathcal{B}$ associated with a pair $(\mathcal{P},s)$ we have
\begin{equation*}
\left(M_\delta^{\mathcal{B}}\right)^n=
\begin{pmatrix}
1&-K&-1&-1\\0&1&0&0\\0&0&1&0\\0&0&0&1\\
\end{pmatrix}
\end{equation*}
Note, however, that the form of $M_\delta^{\mathcal{B}}$ is not determined uniquely and can differ for different singularities of type $\tfrac{1}{n}C$.

\section{Frobenius bases near singular points}\label{s:frob}

In a neighbourhood of every point $t\in\mathbb{P}^1$ (which may be a singular point of $\mathcal{P}$) there exists a certain basis of (possibly multivalued) solutions of $\mathcal{P}=0$, called the \textit{Frobenius basis} at $t$. These are the bases constructed by the classical Frobenius method. For detailed description of the Frobenius method and the associated basis of solutions we refer the reader to \cite{Ince}; here we will only be interested in Frobenius bases at MUM and $\tfrac{1}{n}C$ points.

The \textit{normalized Frobenius basis} $F_{m}$ at a MUM point at 0 is of the form
\begin{equation}
	y_1=f_1,\quad y_2=\tfrac{1}{2\pi i}\cdot\left(f_2+\log(t)f_1\right),\quad
	y_3=\tfrac{1}{(2\pi i)^2}\cdot\left(f_3+\log(t)f_2+\tfrac{1}{2}(\log(t))^2f_1\right),
\end{equation}
\begin{equation*}
	y_4=\tfrac{1}{(2\pi i)^3}\cdot\left(f_4+\log(t)f_3+\tfrac{1}{2}(\log(t))^2f_2+\tfrac{1}{6}(\log(t))^3f_1\right),
\end{equation*}	
where the functions $f_i$ are holomorphic in a neighbourhood of $0$, $f_1(0)=1$ and $f_i(0)=0$ for $i=2,3,4$. Consequently, in this basis the local monodromy around $0$ is given by the matrix
\begin{equation*}
M_0^{F_m}=
\begin{pmatrix}
1&1&\tfrac{1}{2}&\tfrac{1}{6}\\
0&1&1&\tfrac{1}{2}\\
0&0&1&1\\
0&0&0&1
\end{pmatrix}.
\end{equation*}

Let $\mathcal{B}=\mathcal{B}^{M,A}$ be some generalized Doran-Morgan basis, where $M$ is a local monodromy around $0$. Since by Remark \ref{remark:MUM} we know the monodromy matrix $M^{\mathcal{B}}_0$, solving a system of linear equations
\begin{equation*}
T^{\mathcal{B}}_{F_m} M^{\mathcal{B}}_0-M^{F_m} T^{\mathcal{B}}_{F_m}=0
\end{equation*}
in coefficients of the matrix $T^{\mathcal{B}}_{F_m}$, we can deduce that the transition matrix from the generalized Doran-Morgan basis to the normalized Frobenius basis at $0$ is of the form
\begin{equation*}
T^{\mathcal{B}}_{F_m} =\begin{pmatrix}
a_1&0&0&0\\a_2&a_1&0&0\\a_3&\frac{a_1}{2}+a_2&\frac{a_1}{d_\gamma}&0\\a_4&\frac{a_1}{6}+\frac{a_2}{2}+a_3&\frac{a_1+a_2}{d_\gamma}&\frac{a_1}{d_\gamma}
\end{pmatrix},\quad a_i\in\mathbb{C}.
\end{equation*}
It is a remarkable conjecture, motivated by the phenomenon of mirror symmetry, that $a_i\in\mathbb{Q}[\tfrac{\zeta(3)}{(2\pi i)^3}]$. In fact, one expects even stronger result to hold:
\begin{conjecture}\label{conjecture}
	Let $\mathcal{P}$ be a Picard-Fuchs operator of a family of Calabi-Yau threefolds having a MUM point at $0$ and such that the associated local system is irreducible.
	Then there exists $\gamma\in\pi_1(\mathbb{P}^1\setminus\mathcal{S},t_0)$ such that $d_{\gamma}\neq0$ and there is an asymptotic expansion
	\begin{equation*}
	N_\gamma(-v)=\frac{d_\gamma}{6(2\pi i)^{3}}\log^{3}(t)+\frac{p_\gamma}{48\pi i}\log(t)+\frac{a_\gamma}{(2\pi i)^3}\zeta(3)+o(1)
	\end{equation*}
	where $N_\gamma:=M_\gamma-\operatorname{Id}$ and $d_\gamma,a_\gamma,p_\gamma\in\mathbb{Z}$.
\end{conjecture}

This conjecture, although proven only for a special class of hypergeometric operators, has been verified numerically in a large number of cases (see \cite{Hofmann}, \cite{database}, \cite{Chmiel2}). In the case when $\gamma$ is a loop defining local monodromy around a conifold singularity, the integers $d_\gamma$, $p_\gamma$ and $a_\gamma$ are expected to be numerical invariants $H^3$, $c_2.H$ and $c_3$ of the mirror Calabi-Yau manifold (\cite{Candelas-de la Ossa-Green-Prakes}, \cite{Chen-Yang-Yui}). For more general choices of $\gamma$ we lack such a geometric interpretation. Nevertheless, the Conjecture \ref{conjecture} seems to hold even for operators without any conifold singularities (see \cite{Chmiel2}).

The significance of the transition matrix $T^{\mathcal{B}}_{F_m}$, especially in the presence of a conifold singularity, leads us to investigate whether the transition matrix between the Doran-Morgan basis and a local basis at a singularity of type $\tfrac{1}{n}C$ can also be identified.

The \textit{normalized Frobenius basis} $F_c$ at a singularity $s$ of type $\tfrac{1}{n}C$ is
\begin{eqnarray}
	&&y_1=f_1,\quad y_2=f_2,\quad y_3=\tfrac{1}{2\pi i \zeta_n}\left(f_3+f_2\log(t-s)\right),\quad y_4=f_4,
\end{eqnarray}
where the asymptotic behaviour of the functions $f_i$ in a neighbourhood of $s$ is  $f_1(s)=1$, $f_2(t)=(t-s)^{\tfrac{1}{n}}+O((t-s)^{1+\tfrac{1}{n}})$, $f_3(t)=O((t-s)^{1+\tfrac{1}{n}})$ and $f_4(t)=(t-s)^{\tfrac{2}{n}}+O((t-s)^{1+\tfrac{2}{n}})$.

Using this description one easily obtains the following useful fact:

\begin{lemma}\label{lemma_FrCform}
Let $s$ be a singularity of type $\tfrac{1}{n}C$. In the normalized Frobenius basis $F_c$ the local monodromy around $s$ is given in its Jordan form. In particular for $n=1$, resp. $n=2$, it reads
$$M_{s}^{F_c}=\begin{pmatrix}
1&0&0&0\\0&1&1&0\\0&0&1&0\\0&0&0&1\\
\end{pmatrix},
\quad\text{resp.}\quad
M_{s}^{F_c}=\begin{pmatrix}
1&0&0&0\\0&-1&1&0\\0&0&-1&0\\0&0&0&1\\
\end{pmatrix}.$$
\end{lemma}

\begin{definition}
	The solution $f_2$ in the normalized Frobenius basis $F_c$ is called the \textit{conifold period}. It is (up to a scalar) uniquely determined as the generator of $\operatorname{im}(M_s^n-\text{Id})$.
\end{definition}

\section{Modularity theorem}\label{s:mod}

Every smooth projective variety $X$ defined over $\mathbb{Q}$ has an associated $L$-function $L(X,s)$. It is defined in terms of the characteristic polynomials of the Frobenius automorphisms acting on the middle cohomology. The special values $L(X,1)$ and $L(X,2)$ of this function appear naturally in the context of period integrals (section \ref{s:values}) and seem to have strong connections with the coefficients of the transition matrix $T^{F_c}_\mathcal{B}$ (section \ref{s:coeff}).

However, the $L$-series defining $L(X,s)$ converges only for $\operatorname{Re}(s)>\tfrac{5}{2}$ and the only known way of continuing it analytically to the special points $1$ and $2$ is by identifying it with an $L$-series of a \textit{modular form}. For this reason, before we proceed, we take a quick detour and briefly recall basic facts concerning modular forms and their relationship with rigid Calabi-Yau threefolds.

For $N\in\mathbb{N}$ we define the \textit{N-th Hecke subgroup} of $SL(2,\mathbb{Z})$ as $$\Gamma_0(N):=\left\{\begin{pmatrix}
a&b\\c&d
\end{pmatrix}\in SL(2,\mathbb{Z}):c\equiv0 \mod N \right\}$$
\textit{Unrestricted modular form of weight $k\in\mathbb{Z}$ and level $N\in\mathbb{N}$} is a holomorphic function $f$ on the upper half-plane $\mathbb{H}=\left\{z\in\mathbb{C}:\text{Im}(z)>0\right\}$ such that
\begin{equation*}
f\Big( \frac{a\tau+b}{c\tau+d} \Big)=(c\tau+d)^kf(\tau)\quad \textrm{for all}\quad \begin{pmatrix}
a&b\\c&d
\end{pmatrix}\in\Gamma_0(N) 
\end{equation*}
Putting $a=b=d=1$ and $c=0$, this condition gives $f(\tau+1)=f(\tau)$, implying that $f$ has Fourier expansion of the form \begin{equation}\label{eq:F}
f(\tau)=\sum_{n=-\infty}^{\infty}c_nq^n,\quad\textrm{where}\quad q=e^{2\pi i \tau}
\end{equation}
A \textit{modular form} is an unrestricted modular form such that in its Fourier expansion $c_n=0$ for all $n<0$; if additionally $c_0=0$, we call it a \textit{cusp form}.

Each cusp form with Fourier expansion \eqref{eq:F} has an associated Dirichlet series $$L(f,s)=\sum_{n=1}^{\infty}\frac{c_n}{n^s}$$ called the \textit{L-series} of $f$. The $L$-series of a cusp form of weight $k$ converges absolutely for $\text{Re}(s)>\tfrac{k}{2}+1$ and therefore defines a holomorphic function in the half-plane $\left\{z\in\mathbb{C}:\text{Re}(z)>\tfrac{k}{2}+1\right\}$.

It is a classical fact that the $L$-function of a modular form can be extended to an analytic function defined in the entire $\mathbb{C}$. This allows us to define the \textit{special values} of the $L$-function. For a cusp form $f$ of weight $k$ these are $L(f,j)$ for $j=1,\cdots,k-1$. For cusp forms of weight $4$, which are the only ones we shall need, the functional equation for the completed $L$-function implies that $L(f,3)=\tfrac{2\pi^2}{N}L(f,1)$. Thus when referring to the special values we will only mean $L(f,1)$ and $L(f,2)$.

For rigid Calabi-Yau threefolds defined over $\mathbb{Q}$ the two considered kinds of $L$-functions are related by the following modularity theorem (\cite{Goueva-Yui}, Theorem 3):

\begin{theorem}\label{th:mod}
	Let $X$ be a rigid Calabi-Yau threefold definied over $\mathbb{Q}$. Then there exists an integer $N$ and a Hecke eigenform form $f\in\mathcal{S}_k(\Gamma_0(N))$ such that we have the equality of $L$-functions $L(X,s)=L(f,s)$.
\end{theorem}

\section{Monodromy action on the conifold period}\label{s:values}

In the next sections we present explicit results concerning the form of the transition matrix $T^{\mathcal{B}}_{F_c}$ between the Doran-Morgan basis $\mathcal{B}$ and the Frobenius basis $F_c$ at a singularity of type $\tfrac{1}{n}C$. Here we begin with presenting our main motivation for considering it. It comes from results presented in \cite{Chmiel}, which could be vastly generalized if the form of the transition matrix $T^{\mathcal{B}}_{F_c}$ was known.

\medskip

\textbf{In the rest of this paper we work in the following setup}: $\mathcal{P}$ is a Picard-Fuchs operator of order 4 with a MUM point at 0 and a singularity of type $\tfrac{1}{n}C$ at $s$. $M_0$ and $M_s$ are some fixed local monodromies around the respective singularities, $N_0:=M_0-\operatorname{Id}$ and $N_s:=M_s-\operatorname{Id}$. We assume that $d_s:=d_{M_s}\neq 0$. Thus we may construct Doran-Morgan basis $\mathcal{B}^{M_0,M_s^n}$ as at the end of section \ref{s:trans}; we will denote this basis simply by $\mathcal{B}$. We also denote by $F_m$, resp. $F_c$, the normalized Frobenius basis at $0$, resp. $s$.

\medskip

The first step towards identifying the transition matrix $T^{\mathcal{B}}_{F_c}$ was taken in \cite{Chmiel}. There we considered Picard-Fuchs operators $\mathcal{P}$ of families of \textit{double octics}. Double octic is a Calabi-Yau threefold obtained as a resolution of singularities of a double cover of $\mathbb{P}^3$ branched along a union of eight planes. Double octics defined over $\mathbb{Q}$ with small Hodge number $h^{1,2}(X)\leq1$ were completely classified in \cite{Cynk-Kocel}. As already mentioned, the case $h^{1,2}(X)=0$ corresponds to rigid manifolds, while manifolds with $h^{1,2}(X)=1$ deform in one-dimensional families.

Assume that the fiber $X_s$ over a $\tfrac{1}{n}C$ singularity $s$ is birational to a rigid Calabi-Yau manifold $\hat{X_s}$ and let $f$ be the modular form associated to this rigid model $\hat{X_s}$ by the modularity theorem. Let us define $\Lambda_f:=L(f,1)\mathbb{Z}\oplus\tfrac{L(f,2)}{2\pi i}\mathbb{Z}$. Clearly, if the special values are non-zero, this is a lattice in $\mathbb{C}$. It was observed in \cite{Cynk-van Straten} that $\Lambda_f$ is commensurable to the lattice of period integrals
\begin{eqnarray*}
	&&\left\{\int_{\gamma}\omega:\;\gamma\in H_{3}(\hat{X_s},\mathbb{Z})\right\}
\end{eqnarray*}
where $\omega\in H^{3,0}(\hat{X_s})$. This also suggests the connection between the special values $L(f,1)$, $\tfrac{L(f,2)}{2\pi i}$ of the $L$-function and values of period functions of the family $X_t$ evaluated at $s$.

\begin{definition}
Let $f_c$ denote the conifold period near $s$, i.e. a generator of $\operatorname{im}(N_s)$. Then we define
$$\mathcal{L}^0_{\mathcal{P},s}:=\Big\langle\big\{M_0^n(f_c)(s): n\in\mathbb{Z}\big\}\Big\rangle$$
\end{definition}

In \cite{Chmiel} we observed that $\mathcal{L}^0_{\mathcal{P},s}$ contains (some integer multiple of) $\Lambda_f$. In many cases $\mathcal{L}^0_{\mathcal{P},s}\otimes\mathbb{Q}=\Lambda_f\otimes\mathbb{Q}$ or, more precisely, both $\mathcal{L}^0_{\mathcal{P},s}$ and $\Lambda_f$ are lattices and they are commensurable. There are, however, singularities $s$ for which $\operatorname{rank}(\mathcal{L}^0_{\mathcal{P},s})=3$. The additional generators of $\mathcal{L}^0_{\mathcal{P},s}$ do not belong to $\Lambda_f\otimes\mathbb{Q}$ but turn out to be related to certain \textit{additional} integrals on a singular double cover of $\mathbb{P}^3$ which defines $\hat{X_s}$ (as opposed to integrals on $\hat{X_s}$ itself).

We see that both cases are interesting in their own right. If  $\operatorname{rank}(\mathcal{L}^0_{\mathcal{P},s})=2$, we can compute the special values $L(f,1)$ and $L(f,2)$, up to a rational multiplicative constant, using the Picard-Fuchs operator of the considered family of Calabi-Yau threefolds. If $\operatorname{rank}(\mathcal{L}^0_{\mathcal{P},s})=3$, the special values cannot be directly identified but instead we obtain information on the periods of the singular fiber $X_s$.

As such the following question arises:

\begin{question}\label{q1}
When is $\mathcal{L}^0_{\mathcal{P},s}$ a lattice?
\end{question}

The answer to this natural question is connected to our main goal of identification of the transition matrix $T^\mathcal{B}_{F_s}$ as follows. Since both the conifold period $f_c$ and $N_s(-v)$ from the Doran-Morgan basis $\mathcal{B}$ generate $\operatorname{im}(N_s)$, $f_c=\alpha N_s(-v)$ for some $\alpha\in\mathbb{C}^*$. In the normalized Frobenius basis $F_c$ the function $e_1$ is the only solution such that $e_1(s)\neq 0$. Thus we obtain
$$\mathcal{L}_{\mathcal{P},s}^0=\mathbb{Z}\text{-span}\Big(\big\{(e_1)^T\cdot (M_0^{F_c})^n \cdot e_2: n\in\mathbb{Z}\big\}\Big)=\alpha\cdot\mathbb{Z}\text{-span}\Big(\big\{(e_1)^T\cdot T^\mathcal{B}_{F_c}\cdot (M^{\mathcal{B}}_0)^n \cdot e_1: n\in\mathbb{Z}\big\}\Big).$$
Since the form of $M_0^\mathcal{B}$ is known, the rank of $\mathcal{L}^0_{\mathcal{P},s}$ can be determined once the form of the transition matrix $T^\mathcal{B}_{F_c}$ is established. For the case of half-conifold singularities, this is done in Section \ref{s:theorem}.

On the other hand, this connection between $\mathcal{L}_{\mathcal{P},s}^0$ and the transition matrix $T^\mathcal{B}_{F_c}$ can also be used in a following way. From \cite{Chmiel} we know that the special values $L(f,1)$ and $\tfrac{L(f,2)}{2\pi i}$ are elements of $\mathcal{L}_{\mathcal{P},s}^0\otimes\mathbb{Q}$. Since $M^{\mathcal{B}}_0\in GL(4,\mathbb{Q})$ and special values are (expected to be) irrational, it indicates that they could perhaps be identified among the coefficients of the matrix $\alpha T^\mathcal{B}_{F_c}$. This idea is pursued in sections \ref{s:coeff} and \ref{s:14}, where these coefficients are (numerically) identified using this method.

\bigskip

The group $\mathcal{L}^0_{\mathcal{P},s}$ is in an obvious way a subgroup of a more intrinsic object:
$$\mathcal{L}_{\mathcal{P},s}:=\Big\langle\big\{M(f_c)(s): M\in 
Mon(\mathcal{P})\big\}\Big\rangle.$$
In the simplest case of hypergeometric operators, this group is equal to the already introduced $\mathcal{L}^0(\mathcal{P},s)$, as we briefly show.

\begin{proposition}\label{prop:simple}
Assume that $s$ is a conifold singularity and that the monodromy group $Mon(\mathcal{P})$ is generated by $M_0$ and $M_s$. Then $\mathcal{L}^0_{\mathcal{P},s}=\mathcal{L}_{\mathcal{P},s}$.
\end{proposition}

\begin{proof}
Let $\mathcal{B}$ be the Doran-Morgan basis and let $L$ be the lattice spanned by $\mathcal{B}$ inside $\mathcal{S}ol(t_0)$. Since $M_0^\mathcal{B},M_s^\mathcal{B}\in GL(4,\mathbb{Z})$, $L$ is monodromy invariant. Take some $f\in L$. The conifold period $f_c$ generates $\operatorname{im}(N_s)$ and it follows that for any $l\in\mathbb{Z}$ we have the equality $M_s^l(f)=f+af_c$ for some $a\in\mathbb{Z}$. In particular, this applies to any solution of the form $f=M(f_c)$ for some $M\in Mon(\mathcal{P})$.

Now take any $M\in Mon(\mathcal{P})$ and write $M=\displaystyle\prod_{i=1}^mM_0^{k_i}M_s^{l_i}$ for some integers $k_1,\cdots,k_m,l_1,\cdots,l_m\in\mathbb{Z}$. We want to show that $M(f_c)(s)\in\mathcal{L}^0_{\mathcal{P},s}$. For $m=1$ we get $M_0^{k_1}M_s^{l_1}(f_c)(s)=M_0^{k_1}(f_c)(s)\in\mathcal{L}^0_{\mathcal{P},s}$, since the conifold period $f_c$ is holomorphic in a neighbourhood of $s$. For $m>1$ put $N:=\displaystyle\prod_{i=2}^mM_0^{k_i}M_s^{l_i}$. Then 
\begin{align*}
&M(f_c)=\displaystyle\prod_{i=1}^mM_0^{k_i}M_s^{l_i}(f_c)=M_0^{k_1}M_s^{l_1}\left(\displaystyle\prod_{i=2}^mM_0^{k_i}M_s^{l_i}\right)(f_c)\\
&=M_0^{k_1}M_s^{l_1}N(f_c)=M_0^{k_1}(N(f_c)+af_c)=\displaystyle \prod_{i=2}^mM_0^{k'_i}M_s^{l_i}(f_c)+aM_0^{k_1}(f_c),\\
&\text{where}\quad k'_2:=k_1+k_2,\ k'_i:=k_i \ \text{for}\  i>2\quad\text{and}\quad a\in\mathbb{Z}.
\end{align*}
By definition $M_0^{k_1}(f_c)(s)\in\mathcal{L}^0_{\mathcal{P},s}$, and the fact that $\displaystyle \prod_{i=2}^mM_0^{k'_i}M_s^{l_i}(f_c)(s)\in\mathcal{L}^0_{\mathcal{P},s}$ follows by induction on $m$. Hence evaluating the above equality at $s$ yields $M(f_c)(s)\in\mathcal{L}^0_{\mathcal{P},s}$.
\end{proof}

It is natural to ask whether the conclusion of Proposition \ref{prop:simple} holds in general, i.e. whether the inclusion $\mathcal{L}^0_{\mathcal{P},s}\otimes\mathbb{Q}\subset\mathcal{L}_{\mathcal{P},s}\otimes\mathbb{Q}$ can be strict. For example: if the rank of the subgroup $\mathcal{L}^0_{\mathcal{P},s}$ is 2 and it is generated by period integrals of the rigid Calabi-Yau threefold $\hat{X}_s$, can we perhaps still identify the additional integrals of the singular model by considering the larger group $\mathcal{L}_{\mathcal{P},s}$ instead?

\begin{question}\label{q2}
	Does the equality $\mathcal{L}^0_{\mathcal{P},s}\otimes\mathbb{Q}=\mathcal{L}_{\mathcal{P},s}\otimes\mathbb{Q}$ hold?
\end{question}	

\section{Period integrals at a half-conifold singularity}\label{s:theorem}

The aim of this section is to answer Questions \ref{q1} and \ref{q2} \textit{in a context of half-conifold singularity}. This case presents considerably less difficulties than that of a conifold point, since the local space of solutions naturally decomposes into $1$- and $(-1)$-eigenspaces which are determined by the monodromy. As we will see, this allows us to answer the considered questions. The case of an ordinary conifold point seems to be more complex, e.g. we know of no criterion which would allow to \textit{a priori} decide whether $\mathcal{L}^0_{\mathcal{P},s}$ is a lattice in this situation.

\begin{theorem}\label{th:la}
	Assume that $s$ is a half-conifold singularity. Then $\dim(\mathcal{L}_{\mathcal{P},s}\otimes\mathbb{Q})\leq2$.
\end{theorem}

\begin{proof}
To prove this theorem we will determine possible forms of the transition matrix $T^\mathcal{B}_{F_c}$. Then we will use the fact that for all $M\in Mon(\mathcal{P})$ we have
$$
M(f_c)(s)=(e_1)^T\cdot T^\mathcal{B}_{F_c}\cdot M^{\mathcal{B}} \cdot T^{F_c}_\mathcal{B} \cdot e_2
$$
Since $M^{\mathcal{B}}\in GL(4,\mathbb{Q})$, once the transition matrices $T^{F_c}_\mathcal{B}$ and $T^\mathcal{B}_{F_c}$ are identified, the generators of $\mathcal{L}_{\mathcal{P},s}\otimes\mathbb{Q}$ will be determined as well. Since $T^{F_c}_\mathcal{B}=\left(T^\mathcal{B}_{F_c}\right)^{-1}$, we will only consider the form of $T^{F_c}_\mathcal{B}$.

We know that the transition matrix $T^{F_c}_\mathcal{B}$ satisfies $T^{F_c}_\mathcal{B}M_{s}^{F_c}=M_{s}^{\mathcal{B}}T^{F_c}_\mathcal{B}$. In the Doran-Morgan basis $\mathcal{B}$ the monodromy operator $M_{s}^2$ is always given by the same matrix (up to the value of one parameter $K$). However, in general we do not know the form of $M_{s}^{\mathcal{B}}$ itself. Finding possible forms of $M_{s}^{\mathcal{B}}$ will be the first step towards identifying $T^{F_c}_\mathcal{B}$.

As just mentioned, we know that
	\begin{equation*}
	(M_{s}^{\mathcal{B}})^2=
	\begin{pmatrix}
	1&-K&-1&-1\\0&1&0&0\\0&0&1&0\\0&0&0&1\\
	\end{pmatrix},\quad K\in\mathbb{Q}
	\end{equation*}
since the square of the monodromy around a half-conifold singularity is the same as the monodromy around a conifold point. We also know the first column of the matrix $M_{s}^{\mathcal{B}}$:
\begin{equation*}
M_{s}^{\mathcal{B}}=
\begin{pmatrix}
-1&b&c&d\\0&f&g&h\\0&j&k&l\\0&n&o&p\\
\end{pmatrix},
\end{equation*}
because the first element of the basis $\mathcal{B}$ is the conifold period which is an eigenvector of the local monodromy with an eigenvalue $-1$.

Putting those two pieces of information together we obtain a system of equations:
\begin{equation}\label{mateq}
\begin{pmatrix}
0&(f-1)b+jc+dn+K&(k-1)c+bg+do+1&(p-1)d+bh+cl+1\\
0&hn+gj+f^2-1&(f+k)g+ho&(f+p)h+gl\\
0&(f+k)j+ln&gj+k^2+ol-1&(k+p)l+hj\\
0&(f+p)n+oj&(k+p)o+ng&hn+ol+p^2-1\\
\end{pmatrix}
=
\begin{pmatrix}
0&0&0&0\\0&0&0&0\\0&0&0&0\\0&0&0&0\\
\end{pmatrix}
\end{equation}

\noindent The number of unknowns can be reduced using the fact that $\det M_{s}^{\mathcal{B}}=\det M_{s}^{F_c}=1$. For simplicity, let us assume that $fk-gj\neq 0$; the remaining cases are dealt with in a similar manner. Here we obtain
$$p=\frac{flo-gln-hjo+hkn-1}{fk-gj}$$
Resulting system can easily be solved, e.g. using Maple (or, with not much difficulty, by hand), which gives us four possible forms of $M_{s}^{\mathcal{B}}$:

\begin{equation*}
\begin{pmatrix}
-1&\tfrac{(o+Kg+2)c-do-1}{g}&c&d\\
0&Kg+1&g&g\\
0&-K(o+Kg+2)&-Kg-o-1&-Kg-o-2\\
0&oK&o&o+1\\
\end{pmatrix},\quad
\begin{pmatrix}
-1&b&\tfrac{(l+2)d-1}{l}&d\\
0&1&0&0\\
0&lK&l+1&l\\
0&-K(l+2)&-l-2&-l-1\\
\end{pmatrix},$$
$$\begin{pmatrix}
-1&b&c&\tfrac{1}{2}\\
0&1&0&0\\
0&0&1&0\\
0&-2K&-2&-1\\
\end{pmatrix},\quad
\begin{pmatrix}
-1&\tfrac{K}{2}&\tfrac{1}{2}&\tfrac{1}{2}\\
0&-1&0&0\\
0&0&-1&0\\
0&0&0&-1\\
\end{pmatrix}.
\end{equation*}
To find the transition matrix $T^{F_c}_\mathcal{B}$ corresponding to each of this possibilities, we must solve the system of (this time, linear) equations $$T^{F_c}_\mathcal{B}M_{s}^{F_c}-M_{s}^{\mathcal{B}}T^{F_c}_\mathcal{B}=0$$
This way we obtain again four possibilities:
\begin{equation}\label{eq:tr1}
T^{F_c}_\mathcal{B}=\begin{pmatrix}
-\tfrac{-co+do-2c+1}{2g}t_{21}-\tfrac{c-d}{2}t_{41}&-\tfrac{1}{g}t_{23}&t_{13}&-\tfrac{-co+do-2c+1}{2g}t_{24}-\tfrac{c-d}{2}t_{44}\\
t_{21}&0&t_{23}&t_{24}\\
-Kt_{21}-t_{41}&0&-\tfrac{o+Kg+2}{g}t_{23}&-Kt_{24}-t_{44}\\
t_{41}&0&\tfrac{o}{g}t_{23}&t_{44}\\
\end{pmatrix},
\end{equation}
\begin{equation}
T^{F_c}_\mathcal{B}=\begin{pmatrix}
-\tfrac{Kdl-bl}{2l}t_{21}-\tfrac{-2d+1}{2l}t_{31}&-\tfrac{1}{l}t_{33}&t_{13}&-\tfrac{Kdl-bl}{2l}t_{24}-\tfrac{-2d+1}{2l}t_{34}\\
t_{21}&0&0&t_{24}\\
t_{31}&0&0&t_{24}\\
-Kt_{21}-t_{31}&0&-\tfrac{l+2}{l}t_{33}&-Kt_{24}-t_{34}\\
\end{pmatrix},
\end{equation}
\begin{equation}
T^{F_c}_\mathcal{B}=\begin{pmatrix}
(-\tfrac{K}{4}+\tfrac{b}{2})t_{21}+(\tfrac{c}{2}-\tfrac{1}{4})t_{31}&t_{12}&t_{13}&(-\tfrac{K}{4}+\tfrac{b}{2})t_{24}+(\tfrac{c}{2}-\tfrac{1}{4})t_{34}\\
t_{21}&0&0&t_{24}\\
t_{31}&0&0&t_{34}\\
-Kt_{21}-t_{31}&0&2t_{12}&-Kt_{24}-t_{34}\\
\end{pmatrix},
\end{equation}
\begin{equation}
T^{F_c}_\mathcal{B}=\begin{pmatrix}
0&t_{12}&t_{13}&0\\
0&0&t_{23}&0\\
0&0&-Kt_{23}+2t_{12}-t_{43}&0\\
0&0&t_{43}&0\\
\end{pmatrix}.
\end{equation}
for some $t_{ij}\in\mathbb{C}$.

The last matrix cannot be a transition matrix, which leaves us with three possibilities. Take any matrix $Q=(q_{ij})_{i,j=1,\cdots,4}\in GL(4,\mathbb{Q})$. We can directly check that in those three cases the expression $(e_1)^T\cdot T^\mathcal{B}_{F_c}\cdot Q \cdot T^{F_c}_\mathcal{B}\cdot e_2$ equals respectively:

\begin{equation}\label{eq:1}
\tfrac{t_{23}}{2g(t_{24}t_{41}-t_{21}t_{44})}\cdot\Big((-Koq_{21}-(o+2)q_{41}-oq_{31})t_{24}+(g(Kq_{21}+q_{31}+q_{41})+2q_{21})t_{44}\Big),
\end{equation}
\begin{equation}\label{eq:2}
\tfrac{t_{33}}{2l(t_{21}t_{34}-t_{24}t_{31})}\cdot\Big((l(Kq_{21}+q_{31}+q_{41})+2q_{31})t_{24}-2q_{21}t_{34}\Big),
\end{equation}
\begin{equation}\label{eq:3}
\tfrac{t_{12}}{t_{21}t_{34}-t_{24}t_{31}}\cdot\Big(-q_{31}t_{24}+q_{21}t_{34}\Big).
\end{equation}
Thus in all cases $\mathcal{L}_{\mathcal{P},s}\otimes\mathbb{Q}$ is generated by two elements.
\end{proof}

The only obstruction to strengthening the conclusion of Theorem \ref{th:la} to $\dim(\mathcal{L}_{\mathcal{P},s}\otimes\mathbb{Q})=2$ is the possibility that the coefficients $t_{24}$ and $t_{44}$ are linearly dependent over $\mathbb{Q}$. This possibility cannot be excluded on the basis of our current considerations, as it requires the knowledge of exact values of the coefficients of $T^{F_c}_\mathcal{B}$. We present our results in this direction in the following section. Before we do, let us quickly show how Theorem \ref{th:la} allows us to answer Question \ref{q2} as well.

\begin{corollary}\label{cor}
$\mathcal{L}_{\mathcal{P},s}^0\otimes\mathbb{Q}=\mathcal{L}_{\mathcal{P},s}\otimes\mathbb{Q}$
\end{corollary}

\begin{proof}
We present the proof in the case when the transition matrix $T^{F_c}_\mathcal{B}$ is of the form (\ref{eq:tr1}); other cases are treated completely analogously.
 
From the proof of Theorem \ref{th:la} one sees that $\mathcal{L}_{\mathcal{P},s}\otimes\mathbb{Q}$ is generated by $\alpha t_{24}$ and $\alpha t_{44}$, where
$$
\alpha:=\frac{t_{23}}{2g(t_{24}t_{41}-t_{21}t_{44})}\neq 0
$$
We know the form of $M_0^\mathcal{B}$ (see Remark \ref{remark:MUM}):
\begin{equation*}
M_0^\mathcal{B}=\begin{pmatrix}
1&0&0&0\\1&1&0&0\\0&d_s&1&0\\0&0&1&1\\
\end{pmatrix}
\end{equation*}
Thus we can directly compute elements of $\mathcal{L}_{\mathcal{P},s}^0$. Using the formula (\ref{eq:1}) respectively for $Q=M_0^\mathcal{B},\left(M_0^\mathcal{B}\right)^{-1},\left(M_0^\mathcal{B}\right)^2$, we obtain
\begin{eqnarray*}
	\tfrac{1}{\alpha}\cdot M_0(f_c)(s)&=&oK\cdot t_{24}-(gK+2)\cdot t_{44}\\
	\tfrac{1}{\alpha}\cdot M_0^{-1}(f_c)(s)&=&-(oK+2d_s)\cdot t_{24}+(gK+2)\cdot t_{44}\\
	\tfrac{1}{\alpha}\cdot M_0^2(f_c)(s)&=&o(2K+d_s)\cdot t_{24}-\left(2(gK+2)+gd_s\right)\cdot t_{44}
\end{eqnarray*}
Here all coefficients in front of $t_{24}$ and $t_{44}$ are rational.

Note that $\tfrac{1}{\alpha}\cdot\mathcal{L}^0_{\mathcal{P},s}\ni M_0(f_c)(s)+ M_0^{-1}(f_c)(s)=-2d_st_{24}$. By our general assumption $d_s\neq 0$, since it is a necessary condition for the construction of the Doran-Morgan basis $\mathcal{B}$. Thus we conclude that $\alpha t_{24}\in\mathcal{L}^0_{\mathcal{P},s}\otimes\mathbb{Q}$, and consequently that $(gK+2)t_{44},\left(2(gK+2)+gd_s\right)t_{44}\in\tfrac{1}{\alpha}\mathcal{L}^0_{\mathcal{P},s}\otimes\mathbb{Q}$. Hence $\alpha t_{44}\in\mathcal{L}^0_{\mathcal{P},s}\otimes\mathbb{Q}$ as well.
\end{proof}

\section{Transition matrix from the Doran-Morgan basis
	\\
	to the Frobenius basis at a half-conifold point}\label{s:coeff}

In the previous section we have determined the possible forms of the transition matrix $T^{F_c}_\mathcal{B}$ in the case of a half-conifold singularity. Their coefficients are rational linear combinations of six unknown parameters $t_{ij}$. These considerations do not, obviously, give us any information on what the values of the coefficients $t_{ij}$ are. The purpose of this section is to present results concerning values of the parameters $t_{ij}$ in the context of operators associated with families of double octics considered in \cite{Chmiel}.

To be more precise, operators considered here are not identical to those presented in \cite{Chmiel}. There one considered operators which were \textit{actual} Picard-Fuchs operators of the corresponding families of double octics. However, these operators can often be simplified which makes them easier for numerical experiments. The most common example of this phenomena is the situation when an operator $\mathcal{P}$ is a pull-back of some different operator $\mathcal{Q}$ of smaller degree (see \cite{Cynk-van Straten_2}). In this situation it may happen that $\mathcal{P}$ has a conifold singularity but for $\mathcal{Q}$ it becomes a half-conifold one. This provides us with many more examples of half-conifold points than were originally present in operators from \cite{Chmiel}.

At the beginning of this section we must place a single disclaimer. All results presented here were observed numerically. The Frobenius method can be easily implemented in Maple which allows us to compute numerical approximations of the monodromy action on any solution of $\mathcal{P}=0$ (see \cite{vanEnckevort-vanStraten}, \cite{Hofmann}, \cite{Chmiel}). This way we can compute approximations of the transition matrix $T^{F_c}_\mathcal{B}$ as well, and then attempt to numerically identify its coefficients. At present, we do not have proofs of the presented results in the classical sense.

\subsection{Exemplary calculations. Operator 2.17}\ \newline\par
Contrary to the situation at a point of maximal unipotent monodromy, there is not much anticipation on how the transition matrix between the Doran-Morgan basis and the Frobenius basis at a conifold point should look like. Instead of trying to work in the greatest possible generality, which would only cloud the picture, it is perhaps better to begin with an explicit example. We have chosen a particularly simple operator, nevertheless, the phenomena observed for it are common amongst all considered examples.

We will study the operator \textbf{2.17} which is given by
\begin{center}
	$\mathcal{P}=$
	\(\displaystyle {\Theta}^{4}\)
	\mbox{\(\displaystyle\; - \; 2 ^{4}t(2\,\Theta+1)^{2}(8\,{\Theta}^{2}+8\,\Theta+3)\)}
	\mbox{\(\displaystyle\; + \; 2 ^{12}t^{2}(2\,\Theta+1)^{2}(2\,\Theta+3)^{2}\)}
\end{center}
Its singularities can be presented in the form of the so called \textit{Riemann scheme}:
\[\left\{\begin{tabular}{*{3}c}
0& $\frac{1}{256}$& $\infty$\\ 
\hline
0& 0& 1/2\\
0& 1/2& 1/2\\
0& 1/2& 3/2\\
0& 1& 3/2\\
\end{tabular}\right\}\]	
Here each column corresponds to a singularity of the operator and the numbers under the singularity denote roots of the corresponding indicial equation. Thus 0 is a MUM point and $s=\tfrac{1}{256}$ is a half-conifold singularity.

We will denote the normalized Frobenius basis at $s=\tfrac{1}{256}$ by $y_1,y_2,y_3,y_4$, as in Section \ref{s:frob}. This means that $y_1(\tfrac{1}{256})=1$, $y_2=f_c$ is the conifold period, $y_3$ is the only solution in the Frobenius basis containing logarithm and $y_4$ is the unique solution such that $y_4=O(t-\tfrac{1}{256})$. As for the choice of $v$ in the Doran-Morgan basis $\mathcal{B}$, we normalize it so that $v(0)=1$.

The fiber $X_s$ at the half-conifold singularity at $\tfrac{1}{256}$ is birational to a rigid double octic and rank of the group $\mathcal{L}^0_{\mathcal{P},s}$ is 2. It is a lattice generated by $4L(f,1)$ and $4\frac{L(f,2)}{2\pi i}$, where $f$ is the unique modular form of weight 4 and level 8, associated with a smooth model of the singular fiber $X_s$ by the modularity theorem.

Monodromy matrices in the generalized Doran-Morgan basis $\mathcal{B}$ are
\begin{equation*}
M^{\mathcal{B}}_0=
\begin{pmatrix}
1&0&0&0\\1&1&0&0\\0&32&1&0\\0&0&1&1\\
\end{pmatrix}\quad\text{and}\quad
M^{\mathcal{B}}_s=
\begin{pmatrix}
-1&0&\frac{1}{4}&\frac{1}{2}\\0&-1&-\frac{1}{4}&-\frac{1}{4}\\0&32&5&4\\0&-32&-4&-3\\
\end{pmatrix}
\end{equation*}
Solving the system of equations
$$T^{F_c}_\mathcal{B}M_{s}^{F_c}-M_{s}^{\mathcal{B}}T^{F_c}_\mathcal{B}=0$$
in this particular case, we see that
\begin{equation*}
T^{F_c}_\mathcal{B}=
\begin{pmatrix}
a_1&4a_2&a_3&a_4\\
a_5&0&a_2&a_6\\
-8a_1-16a_5&0&-16a_2&-8a_4-16a_6\\
8a_1+8a_5&0&16a_2&8a_4+8a_6\\
\end{pmatrix},\quad a_i\in\mathbb{C}
\end{equation*}
This corresponds to the case (\ref{eq:tr1}) considered in the proof of Theorem \ref{th:la}. Computing approximations of the coefficients $a_i$ using Maple we have obtained the following:

\begin{theorem}\label{ob:approx}
	With a relative error of at most $10^{-195}$ the transition matrix $T^{F_c}_\mathcal{B}$ from the Frobenius basis at $\tfrac{1}{256}$ to the Doran-Morgan basis $\mathcal{B}$ for the operator \textbf{2.17} is given by
	\begin{equation*}
	\begin{pmatrix}\bigskip
	\frac{128L(f,1)a_5+\pi}{128 L(f,2)}\pi i&-\frac{\pi}{128}&\frac{\pi}{256}+(\frac{-1+2\log(2)}{128})i&-\frac{\pi^2 }{256}L(f,1)i\\\bigskip
	a_5&0&-\frac{\pi}{512}&-\frac{\pi}{256}L(f,2)\\\bigskip
	-16a_5-\frac{128L(f,1)a_5+\pi}{16 L(f,2)}\pi i&0&\frac{\pi}{32}&\frac{\pi}{16}L(f,2)+\frac{\pi^2 }{32}L(f,1)i\\
	8a_5+\frac{128L(f,1)a_5+\pi}{16 L(f,2)}\pi i&0&-\frac{\pi}{32}&-\frac{\pi}{32}L(f,2)-\frac{\pi^2 }{32}L(f,1)i\\
	\end{pmatrix}
	\end{equation*}
\end{theorem}

We will now shortly describe how the entries of the transition matrix were identified. Computing the scaling factor for the conifold period $f_c$, which appears in both bases, we find that 
\begin{equation*}
y_2=-\frac{\pi}{128}N_s^2(-v)
\end{equation*}
and so $a_2=-\frac{\pi}{512}$.
For the second solution in the $(-1)$-eigenspace, its representation in $\mathcal{B}$ is numerically identified to be
\begin{equation*}
y_3=\big( \frac{\pi}{256}+\alpha i\big) N_s^2(-v)-\frac{\pi}{512}N_0(N_s^2(-v))+\frac{\pi}{32}N_0^2(N_s^2(-v))-\frac{\pi}{32}v,
\end{equation*}
where $\alpha=\frac{-1+2\log(2)}{128}$.

Thus coefficients of $T^{F_c}_\mathcal{B}$ associated with the $(-1)$-eigenspace seem rather simple. The only exception is the appearance of $\log(2)$ but this term can be eliminated by an appropriate scaling as follows. Changing the variable $t\mapsto t+s$ we may assume $s=0$. After scaling $t\mapsto\alpha t$ we get new normalized Frobenius basis $\{\widetilde{y_1},\dots,\widetilde{y_4}\}$ such that (up to multiplication by powers of $\alpha$) $\widetilde{y_1}(t)=y_1(\alpha t)$, $\widetilde{y_2}(t)=y_2(\alpha t)$, $\widetilde{y_4}(t)=y_4(\alpha t)$ and most
notably $\widetilde{y_3}(t)=y_3(\alpha t)-\tfrac{1}{2\pi i}\log(\alpha)y_2(\alpha t)$. Thus taking $\alpha=8$ we get
\begin{equation*}
\widetilde{y_3}=\big( \frac{\pi}{256}-\frac{1}{128}i\big) N_{\gamma^2}(-v)-\frac{\pi}{512}N_0(N_{\gamma^2}(-v))+\frac{\pi}{32}N_0^2(N_{\gamma^2}(-v))-\frac{\pi}{32}v
\end{equation*}
and the term with $\log(2)$ disappears.

Now we turn to the $1$-eigenspace. As already mentioned, $y_4$ (up to scaling) is uniquely determined: it is the only solution in $O(t-\tfrac{1}{256})$. Comparing coefficients of its representation in the basis $\mathcal{B}$ with the special values $L(f,1)$ and $L(f,2)$, we find
\begin{align*}
y_4=&-\frac{\pi^2 L(f,1)}{256}iN_s^2(-v)-\frac{\pi L(f,2)}{256}N_0(N_s^2(-v))
\\&+\Big(\frac{\pi L(f,2)}{16}+\frac{\pi^2L(f,1)}{32}i \Big)N_0^2(N_s^2(-v))-\Big(\frac{\pi L(f,2)}{32}+\frac{\pi^2L(f,1)}{32}i \Big)v.
\end{align*}
Hence $a_4=-\frac{\pi^2 L(f,1)}{256}i$ and $a_6=-\frac{\pi L(f,2)}{256}$.

Now we consider the solution $y_1$. It is the only solution in $F_c$ with non-zero value at $\frac{1}{256}$. Computing $\mathcal{L}_{\mathcal{P},\tfrac{1}{256}}$ by analytic continuation as in \cite{Chmiel}, we find that $M_0(y_2)(\tfrac{1}{256})=4\frac{L(f,2)}{2\pi i}$. Since
$$-\frac{128}{\pi}M_0(y_2)=M_0(N_s^2(-v))=N_s^2(-v)+N_0(N_s^2(-v)),$$
evaluating at $\frac{1}{256}$ yields
$$-\frac{512}{\pi}\frac{L(f,2)}{2\pi i}=N_0(N_s^2(-v))(\tfrac{1}{256})=e_1^T\cdot T^{\mathcal{B}}_{F_c}\cdot e_2$$
This way we obtain the equation
$$\frac{512}{\pi}\frac{L(f,2)}{2\pi i}=\frac{2L(f,2)}{L(f,2)a_1-\pi iL(f,1)a_5}.$$
Hence
$$a_1=\frac{128L(f,1)a_5+\pi}{128 L(f,2)}\pi i.$$
This is nothing more than using the identity (\ref{eq:1}), which relates values of the conifold period after monodromy with the coefficients of the transition matrix $T^\mathcal{B}_{F_c}$.

Combining all information obtained so far, we see that the transition matrix $T^{F_c}_\mathcal{B}$ indeed has the form as in the Observation \ref{ob:approx}.
Unlike with other coefficients, we were unable to identify
$$a_5=-2.86962386860844996439493221586251437058113644169225...$$

\subsection{General results}\ \newline\par

Similar procedure allows us to identify coefficients of the transition matrix for other operators with half-conifold singularities. In all considered cases the transition matrix has the form
\begin{equation}\label{eq:tr2}
T^{F_c}_\mathcal{B}=\begin{pmatrix}
q_1t_{21}+q_2t_{41}&t_{23}&t_{13}&q_1t_{24}+q_2t_{44}\\
t_{21}&0&q_3t_{23}&t_{24}\\
-Kt_{21}-t_{41}&0&q_4t_{23}&-Kt_{24}-t_{44}\\
t_{41}&0&q_5t_{23}&t_{44}
\end{pmatrix},\quad q_i\in \mathbb{Q}
\end{equation}
i.e. it correspond to the case (\ref{eq:tr1}) listed in the proof of Theorem \ref{th:la}. What we have found, similarly to the case of operator \textbf{2.17} discussed above, is the following:

\begin{theorem}[numerical]\label{th:num}
Let $\mathcal{P}$ be a Picard-Fuchs operator of a family of double octics with a MUM point at 0 and a half-conifold singularity at $s$ with $d_s\neq 0$. Assume that $X_s$ is birational to a rigid double octic $\hat{X_s}$ defined over $\mathbb{Q}$ and let $f$ be the modular form associated to $\hat{X_s}$. Put $K:=\mathbb{Q}[\sqrt{2},i]$.

Then the transition matrix $T^{F_c}_\mathcal{B}$ has the form (\ref{eq:tr2}) and:
\begin{itemize}
\item $\tfrac{1}{2\pi i}t_{23}\in K$;
\item after appropriate scaling $t_{13} \in K$;
\item $t_{24},t_{44}\in\Lambda_f\otimes K$, where $\Lambda_f=L(f,1)\mathbb{Z}\oplus\tfrac{L(f,2)}{2\pi i}\mathbb{Z}$;
\item $t_{41}$ is an element of the field $K(L(f,1),L(f,2),\pi,t_{21})$.
\end{itemize}
\end{theorem}
\noindent The reason we see coefficients from a number field $K$ and not $\mathbb{Q}$ as in the Observation \ref{ob:approx} is the fact that the family in question can be \textit{twisted}, which corresponds to the twist of the associated modular form by a Dirichlet character. This issue, although interesting, is not important here.

Amongst the conditions listed above, the first two are rather natural. The fourth one follows from the previous ones when combined with the fact, established in \cite{Chmiel}, that $M_0(f_c)(s)\in\Lambda_f\otimes K$. The most surprising is the fact that the coefficients of the solution $y_4=O(t-s)$ are elements of $\Lambda_f\otimes K$ as well. Indeed, the fact that $M_0(f_c)(s)\in\Lambda_f\otimes K$ suggest that the special values should appear in the transition matrix $T^{F_c}_\mathcal{B}$ but, as one sees from the formula (\ref{eq:1}), due to the scaling factor involved it does not imply that $t_{24},t_{44}\in\Lambda_f\otimes K$ and only that $t_{24},t_{44}\in\alpha^{-1}\cdot\left(\Lambda_f\otimes K\right)$ for $\alpha:=\tfrac{t_{23}}{2g(t_{24}t_{41}-t_{21}t_{44})}$. Nevertheless, numerical experiments suggest that $\alpha\in K$ for all considered operators.

Precise values of the coefficients $t_{ij}\in\mathbb{C}$ and rational parameters $q_i$ can be found in the Appendix, where transition matrices for all considered half-conifold singularities are presented. In total, for Picard-Fuchs operators of families of double octics there are 16 singularities meeting our criteria. We have also verified the conclusion of Theorem \ref{th:num} for several singularities appearing in Picard-Fuchs operators of families of Schoen's fiber products.

\subsection{Transition matrix between Frobenius bases}\ \newline\par

In this paper we considered three bases of solutions: the Doran-Morgan basis $\mathcal{B}$, the Frobenius basis $F_m$ at a MUM point and the Frobenius basis $F_c$ at a half-conifold point. The expected form of the transition matrix $T^{F_m}_\mathcal{B}$ is well known (see Conjecture \ref{conjecture}). It was the aim of this paper to identify the transition matrix $T^{F_c}_\mathcal{B}$. This opens up the possibility of identifying the transition matrix $T^{F_c}_{F_m}=T^{\mathcal{B}}_{F_m}T^{F_c}_{\mathcal{B}}$ between the two local bases.
	
Of course this equality only makes sense when $d_s\neq 0$. Otherwise, we cannot construct the Doran-Morgan basis $\mathcal{B}$ associated to the singularity at $s$. However, the transition matrix $T^{F_c}_{F_m}$ is well defined even when $d_s=0$. Thus, once the form of the transition matrix $T^{F_c}_{F_m}$ is established for half-conifold points with $d_s\neq0$, it is possible to anticipate how it should look for singularities with $d_s=0$.
	
Consider for example operator \textbf{6.15}:
	\begin{center}
		\(\displaystyle {\Theta}^{4}\)
		\mbox{\(\displaystyle\; - \; 2 ^{4}t(56\,{\Theta}^{4}+16\,{\Theta}^{3}+22\,{\Theta}^{2}+14\,\Theta+3)\)}
		\mbox{\(\displaystyle\; + \; 2 ^{10}t^{2}(308\,{\Theta}^{4}+272\,{\Theta}^{3}+347\,{\Theta}^{2}+174\,\Theta+35)\)}
		\mbox{\(\displaystyle\; - \; 2 ^{18}t^{3}(212\,{\Theta}^{4}+384\,{\Theta}^{3}+473\,{\Theta}^{2}+282\,\Theta+69)\)}
		\mbox{\(\displaystyle\; + \; 2 ^{26}t^{4}(77\,{\Theta}^{4}+232\,{\Theta}^{3}+327\,{\Theta}^{2}+226\,\Theta+62)\)}
		\mbox{\(\displaystyle\; - \; 2 ^{35}t^{5}(\Theta+1)^{2}(7\,{\Theta}^{2}+17\,\Theta+13)\)}
		\mbox{\(\displaystyle\; + \; 2 ^{42}t^{6}(\Theta+1)^{2}(\Theta+2)^{2}\)}
	\end{center}
It has half-conifold singularity at $s=\tfrac{1}{128}$ which yields $d_s=0$. Using the heuristic just described we were able to identify
	
	\begin{equation*}
	T^{F_c}_{F_0}=
	\begin{pmatrix}
	\tfrac{4\beta L(f,2)+\pi^2}{2\pi L(f,1)}i&0&\tfrac{\pi i}{8}&\tfrac{\pi L(f,2)}{8}i\\
	\beta&-\tfrac{\pi i}{8}&\tfrac{2-5\log(2)+\pi i}{16}&\tfrac{\pi^2L(f,1)}{16}\\
	\tfrac{-20\beta L(f,2)-5\pi^2}{48\pi L(f,1)}i&0&-\tfrac{\pi i}{96}&-\tfrac{5\pi L(f2)}{192}i\\
	\tfrac{3\pi^4\beta L(f,1)-156\zeta(3)\beta L(f,2)-39\pi^2\zeta(3)}{48\pi^4L(f,1)}&-\tfrac{5\pi i}{192}&\tfrac{10\pi^2-78\zeta(3)-25\pi^2\log(2)+5\pi^3i}{384\pi^2}&\tfrac{\pi^4L(f,1)-39\zeta(3)L(f,2)}{192\pi^2}\\
	\end{pmatrix}
	\end{equation*}
	
\noindent where $\beta\approx172.622999822288172$. Had we not established the general form of $T^{F_c}_{F_m}$ using singularities with $d_s\neq 0$ and their associated Doran-Morgan bases, the individual entries of this matrix probably would be rather hard to identify numerically.

\section{Singularity of type $\tfrac{1}{4}C$}\label{s:14}

In the previous sections we considered the case of half-conifold singularities. From the point of view of the considered questions, they are considerably simpler to analyse than an ordinary conifold point due to the splitting of the space of solutions $\mathcal{S}ol(t_0)$ into monodromy-invariant eigenspaces. When one tries to apply methods from the proof of Theorem \ref{th:la} to identify the transition matrix $T^{F_c}_{\mathcal{B}}$ for singularities of type $C$, we obtain a matrix depending on 10 parameters, instead of 6 in the half-conifold case. Unlike for singularities of type $\tfrac{1}{2}C$, we were unable to numerically identify these parameters in general.

On the other hand, for singularities of type $\tfrac{1}{n}C$ for $n>2$ one expects to see less independent parameters in the transition matrix $T^{F_c}_{\mathcal{B}}$. Thus we can expect that in this case the coefficients will be easier to identify. However, amongst Picard-Fuchs operators of families of double octics, which provided material for our numerical experiments, there is only one operator with this type of singularity. For this reason we do not have enough data to present numerical observations concerning this type of singularities. Nevertheless, let us end this paper with a specific example of operator \textbf{2.47}. It is the only operator associated with a family of double octics with a singularity of type $\tfrac{1}{n}$C for $n>2$. It is also the only example for which we were able to identify the transition matrix $T^{F_c}_\mathcal{B}$ completely.

Operator \textbf{2.47} is given by
	\begin{center}
	$\mathcal{P}$\ =\ 
	\(\displaystyle {\Theta}^{4}\)
	\mbox{\(\displaystyle\; - \; 2 ^{4}t(3072\,{\Theta}^{4}+5120\,{\Theta}^{3}+3904\,{\Theta}^{2}+1344\,\Theta+169)\)}\\
	\mbox{\(\displaystyle\; + \; 2 ^{23}t^{2}(4\,\Theta+3)(24\,{\Theta}^{3}+62\,{\Theta}^{2}+49\,\Theta+9)\)}
	\mbox{\(\displaystyle\; - \; 2 ^{34}t^{3}(4\,\Theta+1)(4\,\Theta+3)(4\,\Theta+7)(4\,\Theta+9)\)}
\end{center}
and has the Riemann scheme
\[\left\{\begin{tabular}{*{3}c}
0& $\frac{1}{16384}$& $\infty$\\ 
\hline
0& 0& 1/4\\
0& 1/4& 3/4\\
0& 1/4& 7/4\\
0& 1/2& 9/4\\
\end{tabular}\right\}\]	
Monodromy matrices around $s=\frac{1}{16384}$ are
\begin{equation*}
M_s^{F_c}=\begin{pmatrix}
1&0&0&0\\0&i&1&0\\0&0&i&0\\0&0&0&-1\\
\end{pmatrix},\quad
M_s^\mathcal{B}=\begin{pmatrix}
i&0&i/4&0\\0&-i&0&i/4\\0&8i&-i&2i\\0&-8i&2i&3i\\
\end{pmatrix}
\end{equation*}
and consequently
\begin{equation*}
T^{F_c}_{\mathcal{B}}=\begin{pmatrix}
a_1&a_2&a_3&a_4\\ia_1&0&ia_2&-ia_4\\-4(1+i)a_1&0&-4ia_2&-4(1-i)a_4\\4(1+i)a_1&0&8ia_2&4(1-i)a_4\\
\end{pmatrix}
\end{equation*}
Analogously to other examples, here we have $a_2=\tfrac{\pi\sqrt{2}}{256}\cdot(-1+i)$ and $a_4=-\tfrac{\pi L(f,2)}{256}$. We were also able to identify $a_3=\tfrac{\pi+(6\log(2)-4)i}{512}\cdot\sqrt{2}(1+i)$. From the fact that $M_0(N_s^4(-v))(s)=\sqrt{2}\tfrac{L(f,2)}{2\pi i}$, we deduce $a_1=\tfrac{\pi^2}{128L(f,2)}i$.

Thus in this case the transition matrix is identified completely and we obtain the following theorem:

\begin{theorem}
With a relative error of at most $10^{-195}$ the transition matrix $T^{F_c}_\mathcal{B}$ from the Frobenius basis at $\tfrac{1}{16384}$ to the Doran-Morgan basis $\mathcal{B}$ for the operator \textbf{2.47} is given by
\begin{equation*}
T^{F_c}_{\mathcal{B}}=\begin{pmatrix}\bigskip
\tfrac{\pi^2}{128L(f,2)}\zeta_8^2&\tfrac{\pi}{128}\zeta_8^3&\tfrac{\pi+(6\log(2)-4)i}{256}\zeta_8&-\tfrac{\pi}{256}L(f,2)\\\bigskip
\tfrac{\pi^2}{128L(f,2)}&0&\tfrac{\pi}{128}\zeta_8^5&\tfrac{\pi}{256}L(f,2)\zeta_8^2\\\bigskip
\tfrac{\sqrt{2}\pi^2}{32L(f,2)}\zeta_8^7&0&\tfrac{\pi}{32}\zeta_8&\tfrac{\pi\sqrt{2}}{64}L(f,2)\zeta_8^7\\\bigskip
\tfrac{\sqrt{2}\pi^2}{32L(f,2)}\zeta_8^3&0&\tfrac{\pi}{16}\zeta_8^5&\tfrac{\pi\sqrt{2}}{64}L(f,2)\zeta_8^3\\
\end{pmatrix}
\end{equation*}
where $\zeta_8=\tfrac{1+i}{\sqrt{2}}$ is a primitive eighth root of unity.
\end{theorem}

Note that $L(f,1)$ does not appear in $T^{F_c}_{\mathcal{B}}$; this is due to proportionality of special values: $L(f,1)=8\tfrac{L(f,2)}{2\pi}$. It is interesting that in some sense one could predict that the special values will be proportional from the fact that only four unknown parameters appear in the transition matrix. Geometrically, the singular fiber over $\tfrac{1}{16384}$ is birational to the unique rigid double octic with complex multiplication.

\newpage

\section*{Appendix}

In this Appendix, we collect explicit forms of the transition matrices $T^{F_c}_{\mathcal{B}}$ for Picard-Fuchs operators of families of double octics. This provides a refinement of the Theorem \ref{th:num}. The numbering of the operators follows that of the online database \cite{database}. The numbers in the bracket describe the modular form $f$ associated with the smooth, rigid model of the singular fiber $X_s$ (see \cite{Meyer}); the first number denotes the level of $f$.

\subsection*{Operator \textbf{1.3},  half-conifold singularity at $s=\tfrac{1}{1024}$ (8/1)}\ \newline\par

\bigskip

$T^{F_c}_{\mathcal{B}}=\begin{bmatrix}
\alpha&-\tfrac{\pi}{256}\sqrt{2}&
\tfrac{\pi}{512}\sqrt{2}+\tfrac{(3\ln{2}-1)i}{256}\sqrt{2}&
-\tfrac{\pi}{512}\sqrt{2}L(f,2)\\
\beta&0&
-\tfrac{\pi}{1024}\sqrt{2}&
-\tfrac{\pi^2}{1024}\sqrt{2}L(f,1)\\
-8\beta-8\alpha&0&
\tfrac{\pi}{128}\sqrt{2}&
\tfrac{\pi^2}{128}\sqrt{2}L(f,1)+\tfrac{\pi i}{64}\sqrt{2}L(f,2)\\
8\alpha&0&
-\tfrac{\pi}{128}\sqrt{2}&
-\tfrac{\pi i}{64}\sqrt{2}L(f,2)\\
\end{bmatrix}$

\bigskip

\noindent where
$\beta\approx-6.08158309165313377575$
and
$\alpha:=\tfrac{256\beta L(f,2)+\pi^2\sqrt{2}}{128\pi L(f,1)}\cdot i$

\bigskip

\subsection*{Operator \textbf{2.17},  half-conifold singularity at $s=\tfrac{1}{256}$ (8/1)}\ \newline\par

\bigskip

$T^{F_c}_{\mathcal{B}}=\begin{bmatrix}
\alpha&-\tfrac{\pi}{128}&
\tfrac{\pi}{256}+\tfrac{(2\ln{2}-1)i}{128}&
-\tfrac{\pi^2 i}{256}L(f,1)\\
\beta&0&
-\tfrac{\pi}{512}&
-\tfrac{\pi}{256}L(f,2)\\
-16\beta-8\alpha&0&
\tfrac{\pi}{32}&
\tfrac{\pi^2 i}{32}L(f,1)+\tfrac{\pi}{16}L(f,2)\\
8\beta+8\alpha&0&
-\tfrac{\pi}{32}&
-\tfrac{\pi^2 i}{32}L(f,1)-\tfrac{\pi}{32}L(f,2)\\
\end{bmatrix}$

\bigskip

\noindent where
$\beta\approx-2.86962386860844996439$
and
$\alpha:=\tfrac{128\beta L(f,1)+\pi}{128 L(f,2)}\cdot \pi i$

\bigskip

\subsection*{Operator \textbf{5.47},  half-conifold singularity at $s=-\tfrac{1}{16}$ (32/2)}\ \newline\par

\bigskip

$T^{F_c}_{\mathcal{B}}=\begin{bmatrix}
\alpha&
-\tfrac{\pi i}{864}&
\tfrac{4\ln{2}-3}{2592}-\tfrac{\pi i}{128}&
-\tfrac{\pi^2 i}{5184}L(f,1)\\
\beta&0&
-\tfrac{\pi i}{10368}&
-\tfrac{\pi}{5184}L(f,2)\\
-240\beta-24\alpha&
0&
\tfrac{5\pi i}{216}&
\tfrac{\pi^2 i}{216}L(f,1)+\tfrac{5\pi}{108}L(f,2)\\
144\beta+24\alpha&
0&
-\tfrac{7\pi i}{432}&
-\tfrac{\pi^2 i}{216}L(f,1)-\tfrac{\pi}{36}L(f,2)\\
\end{bmatrix}$

\bigskip

\noindent where
$\beta\approx-0.16069457023018073379$
and
$\alpha:=\tfrac{96\beta L(f,1)+\pi}{96 L(f,2)}\cdot \pi i$

\bigskip

\subsection*{Operator \textbf{5.48}, half-conifold singularity at $s=\tfrac{1}{256}$ (32/2)}\ \newline\par

\bigskip

$T^{F_c}_{\mathcal{B}}=\begin{bmatrix}
\alpha&-\tfrac{\pi}{384}&
\tfrac{\pi}{768}+\tfrac{(10\ln{2}-3)i}{1152}&
-\tfrac{\pi i}{2304}L(f,2)\\
\beta&0&
-\tfrac{\pi}{2304}&
-\tfrac{\pi^2}{18432}L(f,1)\\
-48\beta-12\alpha&0&
\tfrac{\pi}{48}&
\tfrac{\pi^2}{384}L(f,1)+\tfrac{\pi i}{192}L(f,2)\\
48\beta+12\alpha&0&
-\tfrac{5\pi}{192}&
\tfrac{\pi^2}{384}L(f,1)-\tfrac{\pi i}{192}L(f,2)\\
\end{bmatrix}$

\bigskip

\noindent where
$\beta\approx-7.51426632604220544682$
and
$\alpha:=\tfrac{192\beta L(f,2)+\pi^2}{24\pi L(f,1)}\cdot i$

\bigskip

\subsection*{Operator \textbf{5.50},  half-conifold singularity at $s=\tfrac{1}{64}$ (6/1)}\ \newline\par

\bigskip

$T^{F_c}_{\mathcal{B}}=\begin{bmatrix}
\alpha&-\tfrac{\pi}{256}\sqrt{2}&
\tfrac{\pi}{512}\sqrt{2}+\tfrac{(9\ln{2}-4)i}{1024}\sqrt{2}&
-\tfrac{9 \pi i}{2048}\sqrt{2}L(f,2)\\
\beta&0&
-\tfrac{\pi}{2048}\sqrt{2}&
-\tfrac{5\pi^2}{4096}\sqrt{2}L(f,1)\\
-80\beta-16\alpha&0&
\tfrac{5\pi}{128}\sqrt{2}&
\tfrac{25\pi^2}{256}\sqrt{2}L(f,1)+\tfrac{9\pi i}{128}\sqrt{2}L(f,2)\\
48\beta+16\alpha&0&
-\tfrac{5\pi}{16}\sqrt{2}&
\tfrac{15\pi^2}{256}\sqrt{2}L(f,1)-\tfrac{9\pi i}{128}\sqrt{2}L(f,2)\\
\end{bmatrix}$

\bigskip

\noindent where
$\beta\approx-12.76773602982673203281$
and
$\alpha:=\tfrac{1152\beta L(f,2)+\pi^2\sqrt{2}}{320 \pi L(f,1)}\cdot i$

\bigskip

\subsection*{Operator \textbf{5.90},  half-conifold singularity at $s=\tfrac{1}{256}$ (8/1)}\ \newline\par

\bigskip

$T^{F_c}_{\mathcal{B}}=\begin{bmatrix}
\beta+\alpha&\tfrac{\pi i}{384}&
-\tfrac{\pi i}{2304}+\tfrac{10\ln{2}-3}{1152}&
-\tfrac{\pi^2}{2304}L(f,1)-\tfrac{\pi i}{1152}L(f,2)\\
-\tfrac{1}{2}\beta+\tfrac{1}{2}\alpha&0&
\tfrac{\pi i}{2304}&
\tfrac{\pi^2}{4608}L(f,1)+\tfrac{\pi i}{2304}L(f,2)\\
-4\alpha&0&
-\tfrac{\pi i}{576}&
-\tfrac{\pi i}{288}L(f,2)\\
24\alpha&0&
\tfrac{\pi i}{64}&
\tfrac{\pi i}{48}L(f,2)\\
\end{bmatrix}$

\bigskip

\noindent where
$\beta\approx-0.92346087106375608614$
and
$\alpha:=\tfrac{-384\beta L(f,2)+\pi^2}{192\pi L(f,1)}\cdot i$

\bigskip

\subsection*{Operator \textbf{5.114},  half-conifold singularity at $s=\tfrac{1}{1024}$ (12/1)}\ \newline\par

\bigskip

$T^{F_c}_{\mathcal{B}}=\begin{bmatrix}
\alpha&-\tfrac{\pi}{512}&
-\tfrac{\pi}{1024}+\tfrac{(4\ln{2}-1)i}{512}&
-\tfrac{3\pi i}{8192}L(f,2)\\
\beta&0&
-\tfrac{\pi}{2048}&
-\tfrac{5\pi^2}{4096}L(f,1)\\
-16\beta-8\alpha&0&
\tfrac{\pi}{128}&
\tfrac{\pi^2}{512}L(f,1)+\tfrac{3\pi i}{1024}L(f,2)\\
16\beta+8\alpha&0&
-\tfrac{3\pi}{256}&
-\tfrac{\pi^2}{512}L(f,1)-\tfrac{3\pi i}{1024}L(f,2)\\
\end{bmatrix}$

\bigskip

\noindent where
$\beta\approx-2.53398911237591330292$
and
$\alpha:=\tfrac{192\beta L(f,2)+\pi^2}{64\pi L(f,1)}\cdot i$

\bigskip

\subsection*{Operator \textbf{6.11},  half-conifold singularity at $s=\tfrac{1}{1024}$ (32/2)}\ \newline\par

\bigskip

$T^{F_c}_{\mathcal{B}}=\begin{bmatrix}
\beta+\alpha&
\tfrac{\pi i}{320}&
-\tfrac{\pi i}{1600}+\tfrac{54\ln{2}-15}{4800}&
\tfrac{\pi^2i}{12800}L(f,1)-\tfrac{\pi}{3200}L(f,2)\\
-\tfrac{1}{2}\beta+\alpha&0&\tfrac{3\pi i}{3200}&
\tfrac{\pi^2i}{12800}L(f,1)+\tfrac{\pi}{6400}L(f,2)\\
-\beta&0&\tfrac{\pi i}{1600}&
\tfrac{\pi}{3200}L(f,2)\\
-10\beta+40\alpha&0&\tfrac{3\pi i}{80}&
\tfrac{\pi^2i}{320}L(f,1)+\tfrac{\pi}{320}L(f,2)\\
\end{bmatrix}$

\bigskip

\noindent where
$\beta\approx-7.18005014181627297598$
and
$\alpha:=\tfrac{-20\beta L(f,1)+\pi}{80L(f,2)}\cdot \pi i$

\bigskip

\subsection*{Operator \textbf{6.15},  half-conifold singularity at $s=\tfrac{1}{64}$ (8/1)}\ \newline\par

\bigskip

$T^{F_c}_{\mathcal{B}}=\begin{bmatrix}
\beta+\alpha&
\tfrac{\pi i}{384}&
-\tfrac{\pi i}{1152}+\tfrac{8\ln{2}-3}{1152}&
-\tfrac{\pi^2}{1152}L(f,1)+\tfrac{\pi i}{1152}L(f,2)\\
-\tfrac{1}{4}\beta+\tfrac{1}{2}\alpha&0&-\tfrac{\pi i}{3072}&
\tfrac{\pi^2}{4608}L(f,1)+\tfrac{\pi i}{2304}L(f,2)\\
2\beta-6\alpha&0&-\tfrac{4\pi i}{1152}&
-\tfrac{\pi^2}{576}L(f,1)-\tfrac{\pi i}{192}L(f,2)\\
-48\beta+144\alpha&0&\tfrac{17\pi i}{192}&
\tfrac{\pi^2}{24}L(f,1)-\tfrac{\pi i}{8}L(f,2)\\
\end{bmatrix}$

\bigskip

\noindent where
$\beta\approx-0.92346087106375608614$
and
$\alpha:=\tfrac{-384\beta L(f,2)+\pi^2}{384\pi L(f,1)}\cdot i$

\bigskip

\subsection*{Operator \textbf{8.53},  half-conifold singularity at $s=\tfrac{1}{64}$ (8/1)}\ \newline\par

\bigskip

$T^{F_c}_{\mathcal{B}}=\begin{bmatrix}
\alpha&-\tfrac{\pi}{320}&\tfrac{\pi}{640}+\tfrac{(13\ln{2}-5)i}{1600}&-\tfrac{\pi^2i}{800}L(f,1)\\
\beta&0&-\tfrac{\pi}{3200}&-\tfrac{\pi}{1600}L(f,2)\\
-160\beta-20\alpha&0&-\tfrac{\pi}{20}&\tfrac{\pi}{10}L(f,2)+\tfrac{\pi^2i}{40}L(f,1)\\
120\beta+20\alpha&0&-\tfrac{7\pi}{160}&\tfrac{3\pi}{40}L(f,2)-\tfrac{\pi^2i}{40}L(f,1)\\
\end{bmatrix}$

\bigskip

\noindent where
$\beta\approx-0.15770778182092469784$
and
$\alpha:=\tfrac{640\beta L(f,1)+\pi}{320L(f2)}\cdot\pi i$

\bigskip

\subsection*{Operator \textbf{12.2},  half-conifold singularity at $s=\tfrac{1}{64}$ (8/1)}\ \newline\par

\bigskip

$T^{F_c}_{\mathcal{B}}=\begin{bmatrix}
\beta+\alpha&\tfrac{\pi}{128}&
\tfrac{9\ln{2}-4}{512}&
\tfrac{\pi^2 i}{256}L(f,1)-\tfrac{\pi}{256}L(f,2)\\
-\tfrac{1}{2}\beta+\tfrac{1}{2}\alpha&0&
\tfrac{\pi i}{512}&
\tfrac{\pi^2 i}{512}L(f,1)+\tfrac{\pi}{10512}L(f,2)\\
-4\alpha&0&
-\tfrac{\pi i}{128}&
-\tfrac{\pi^2 i}{64}L(f,1)\\
-8\beta+24\alpha&0&
\tfrac{5\pi i}{64}&
\tfrac{3\pi^2 i}{32}L(f,1)-\tfrac{\pi}{32}L(f,2)\\
\end{bmatrix}$

\bigskip

\noindent where
$\beta\approx-3.76925393189203814003$
and
$\alpha:=\tfrac{128\beta L(f,1)-\pi}{128 L(f,2)}\cdot \pi i$

\bigskip

\subsection*{Operator \textbf{12.3},  half-conifold singularity at $s=\tfrac{1}{16}$ (8/1)}\ \newline\par

\bigskip

$T^{F_c}_{\mathcal{B}}=\begin{bmatrix}
\alpha&-\tfrac{\pi}{64}&
\tfrac{\pi}{128}+\tfrac{(5\ln{2}-4)i}{256}&
-\tfrac{\pi^2 i}{64}L(f,1)\\
\beta&0&
-\tfrac{\pi}{512}&
\tfrac{\pi}{128}L(f,2)\\
-160\beta-16\alpha&0&
\tfrac{5\pi}{16}&
\tfrac{\pi^2 i}{4}L(f,1)+\tfrac{5\pi}{4}L(f,2)\\
112\beta+16\alpha&0&
-\tfrac{\pi}{4}&
-\tfrac{\pi^2 i}{4}L(f,1)-\tfrac{7\pi}{88}L(f,2)\\
\end{bmatrix}$

\bigskip

\noindent where
$\beta\approx-1.23158098122767942798$
and
$\alpha:=\tfrac{256\beta L(f,1)+\pi}{128L(f,2)}\cdot \pi i$

\bigskip

\subsection*{Operator \textbf{12.4},  half-conifold singularity at $s=\tfrac{1}{16}$ (8/1)}\ \newline\par

\bigskip

$T^{F_c}_{\mathcal{B}}=\begin{bmatrix}
\alpha&-\tfrac{\pi}{224}&
\tfrac{\pi}{448}-\tfrac{(12\ln{2}+7)i}{1568}&
-\tfrac{\pi i}{392}L(f,2)\\
\beta&0&
-\tfrac{\pi}{3136}&
\tfrac{\pi^2}{1568}L(f,1)\\
-560\beta-28\alpha&0&
\tfrac{5\pi}{28}&
\tfrac{5\pi^2}{14}L(f,1)+\tfrac{\pi i}{14}L(f,2)\\
392\beta+28\alpha&0&
-\tfrac{15\pi}{112}&
-\tfrac{\pi^2}{4}L(f,1)-\tfrac{\pi i}{14}L(f,2)\\
\end{bmatrix}$

\bigskip

\noindent where
$\beta\approx-0.08007994775244181798$
and
$\alpha:=\tfrac{1792\beta L(f,2)+\pi^2}{224\pi L(f,1)}\cdot i$

\bigskip

\subsection*{Operator \textbf{12.7},  half-conifold singularity at $s=-\tfrac{1}{512}$ (8/1)}\ \newline\par

\bigskip

$T^{F_c}_{\mathcal{B}}=\begin{bmatrix}
\alpha&\tfrac{\pi i}{512}&
\tfrac{3\ln{2}-1}{512}&
\tfrac{\pi i}{2048}\sqrt{2}L(f,2)\\
\beta+\tfrac{1}{4}\alpha&0&
\tfrac{\pi i}{1024}&
\tfrac{\pi^2}{4096}\sqrt{2}L(f,1)+\tfrac{\pi i}{8192}\sqrt{2}L(f,2)\\
-4\alpha&0&
0&
-\tfrac{\pi^ i}{512}\sqrt{2}L(f,2)\\
16\beta+8\alpha&0&
\tfrac{5\pi i}{256}&
\tfrac{3\pi^2}{256}\sqrt{2}L(f,1)-\tfrac{\pi i}{256}\sqrt{2}L(f,2)\\
\end{bmatrix}$

\bigskip

\noindent where
$\beta\approx6.59772907683923133217$
and
$\alpha:=\tfrac{256\beta L(f,2)+\pi^2\sqrt{2}}{128\pi L(f,1)}\cdot i$

\bigskip

\subsection*{Operator \textbf{12.7},  half-conifold singularity at $s=-\tfrac{1}{256}$ (8/1)}\ \newline\par

\bigskip

$T^{F_c}_{\mathcal{B}}=\begin{bmatrix}
\beta+\alpha&-\tfrac{\pi}{896}&
\tfrac{\pi}{12544}+\tfrac{(26\ln{2}-7)i}{6272}&
-\tfrac{\pi^2 i}{12544}L(f,1)+\tfrac{\pi}{3136}L(f,2)\\
-\tfrac{1}{4}\beta+\alpha&0&
\tfrac{5\pi}{25088}&
-\tfrac{\pi^2 i}{12544}L(f,1)-\tfrac{\pi}{12544}L(f,2)\\
28\beta-168\alpha&0&
-\tfrac{\pi}{32}&
\tfrac{3\pi^2 i}{224}L(f,1)+\tfrac{\pi}{112}L(f,2)\\
-42\beta+244\alpha&0&
\tfrac{5\pi}{112}&
-\tfrac{\pi^2 i}{56}L(f,1)-\tfrac{3\pi}{224}L(f,2)\\
\end{bmatrix}$

\bigskip

\noindent where
$\beta\approx-2.40508365817870201311$
and
$\alpha:=\tfrac{-224\beta L(f,1)+\pi}{896L(f,2)}\cdot \pi i$

\bigskip

\subsection*{Operator \textbf{12.8},  half-conifold singularity at $s=-\tfrac{1}{256}$ (32/2)}\ \newline\par

\bigskip

$T^{F_c}_{\mathcal{B}}=\begin{bmatrix}
\beta+\alpha&\tfrac{\pi}{640}&
\tfrac{3\pi}{6400}+\tfrac{(18\ln{2}-5)i}{3200}&
\tfrac{\pi^2}{25600}L(f,1)-\tfrac{\pi i}{6400}L(f,2)\\
-\tfrac{1}{2}\beta+\tfrac{1}{2}\alpha&0&
\tfrac{\pi}{3200}&
-\tfrac{\pi^2}{51200}L(f,1)-\tfrac{\pi i}{12800}L(f,2)\\
\tfrac{1}{4}\beta-\tfrac{3}{8}\alpha&0&
-\tfrac{\pi}{5120}&
\tfrac{\pi^2}{102400}L(f,1)+\tfrac{3\pi i}{56200}L(f,2)\\
-60\beta+80\alpha&0&
\tfrac{3\pi}{64}&
-\tfrac{3\pi^2}{1280}L(f,1)-\tfrac{\pi i}{80}L(f,2)\\
\end{bmatrix}$

\bigskip

\noindent where
$\beta\approx-1.28003826199452498547$
and
$\alpha:=\tfrac{-160\beta L(f,2)+\pi^2}{40\pi L(f,1)}\cdot i$
 
\end{document}